\documentclass[12pt]{article}
\usepackage{amsmath}
\usepackage{amsmath,amssymb, amsthm, latexsym, amsfonts, epsfig, color,graphicx,}
\usepackage{mathrsfs}
\usepackage{indentfirst}
\usepackage{chngcntr}
\usepackage{bbm}
\usepackage[blocks]{authblk} 
\usepackage{newtxmath} 
\usepackage[colorlinks=True,linkcolor=blue,anchorcolor=blue,citecolor=blue,CJKbookmarks=true]{hyperref}
\usepackage{geometry}
\usepackage{verbatim}
\usepackage{amsfonts}
\DeclareMathOperator{\supp}{supp}

\begin{document}
\renewcommand{\a}{\alpha}
\newcommand{\D}{\Delta}
\newcommand{\ddt}{\frac{d}{dt}}
\counterwithin{equation}{section}
\newcommand{\e}{\epsilon}
\newcommand{\eps}{\varepsilon}
\newtheorem{theorem}{Theorem}[section]
\newtheorem{proposition}{Proposition}[section]
\newtheorem{lemma}[theorem]{Lemma}
\newtheorem{remark}[theorem]{Remark}
\newtheorem{example}{Example}[section]
\newtheorem{definition}{Definition}[section]
\newtheorem{corollary}[theorem]{Corollary}
\makeatletter
\newcommand{\rmnum}[1]{\romannumeral #1}
\newcommand{\Rmnum}[1]{\expandafter\@slowromancap\romannumeral #1@}
\makeatother

\title{\bf Fractal Uncertainty and Quantitative Uniqueness for the Fourier–Bessel Transform}
\author{
    Xingyu Zhao, Longben Wei, Zhiwen Duan
}
\date{}
\maketitle
\begin{abstract} 
We establish quantitative uniqueness and a fractal uncertainty principle for the Fourier–Bessel transform. In arbitrary dimension, we prove a quantitative uniqueness estimate on relatively dense sets for functions whose Fourier–Bessel transforms decay according to a quasi-analytic weight. In dimension one, if \(X\subset[0,1]\) and \(Y\subset[a,a+h^{-1}]\) are \(\delta\)-regular on the relevant scales, then, for \(a\geq a_0h^{-1}\),

\[
\operatorname{supp}\mathcal H_\nu f\subset Y
\quad\Longrightarrow\quad
\|\mathbf 1_Xf\|_{L^2_\nu}
\leq Ch^\beta\|f\|_{L^2_\nu}.
\]

The lack of translation invariance prevents a direct application of the classical Fourier argument. We overcome this by constructing damping functions adapted to translated regular sets and combining Beurling–Malliavin multipliers with large-argument Bessel asymptotics and a Bourgain–Dyatlov multiscale iteration.

\end{abstract}
\noindent{\bf 2020 Mathematics Subject Classification:} 43A32; 28A80.
\newline 
\noindent{\bf Key words:} Multidimensional
Fourier–Bessel transform, Quantitative unique continuation, Fractal uncertainty principle
\section{Introduction}
In their seminal article \cite{BD}, Bourgain and Dyatlov established a new variant of the uncertainty principle for the Fourier transform in one dimension, known as the fractal uncertainty principle (FUP). It asserts that no function can be concentrated on a fractal set in both position and frequency domains. 
\par To state the fractal uncertainty principle precisely, we first fix the convention for the Fourier transform. For \( f \in L^1(\mathbb{R}) \), define
\begin{equation}
\hat{f}(\xi)=\mathscr{F}f(\xi)=\int_{\mathbb{R}}e^{-2\pi i x\xi}f(x)\,dx.
\end{equation}
A useful feature of this normalization is that \(\mathscr{F}\) extends to a unitary operator on \(L^2(\mathbb{R})\). We also recall the notion of Ahlfors–David (AD) regularity.
\begin{definition}
Let $X\subset \mathbb{R}$ be a nonempty closed set and $\delta \in[0,1]$, $C_R\geq1$, $0\leq \alpha_0\leq\alpha_1\leq \infty$. We say that $X$ is $\delta$-regular with constant $C_R$ on scales $\alpha_0$
to $\alpha_1$ if there exists a Borel measure $\mu_X$ on $\mathbb{R}$ such that
\par $\bullet$ $\mu_X$ is supported on $X$, that is $\mu_X(\mathbb{R}\backslash X)=0$;

 $\bullet$ for each interval $I$ of size $|I|\in [\alpha_0, \alpha_1]$, we have $\mu_X(I)\leq C_R |I|^{\delta}$;
 
  $\bullet$ if additionally $I$ is centered at a point in $X$, then $\mu_X(I)\geq C_R^{-1}|I|^{\delta}$.

\end{definition}
We now state Bourgain and Dyatlov’s theorem.
\begin{theorem}(Bourgain-Dyatlov). Let $0\leq \delta<1$, $C_R\geq 1$, $h\leq 1$, and assume that
\par $\bullet$ \(X\subset[-1,1]\)  is $\delta$-regular with constant $C_R$ on scales $h$ to 1; and

$\bullet$ \(Y\subset[-h^{-1},h^{-1}]\) is $\delta$-regular with constant $C_R$ on scales 1 to $h^{-1}$.\\
\\
Then there exist \(\beta,C>0\) depending only on $\delta$, $C_R$ such that for all \(f\in L^2(\mathbb{R})\)
\begin{equation}
\operatorname{supp}\hat{f}\subset Y\Longrightarrow \lVert f \mathbf{1}_{X}\rVert_{L^2(\mathbb{R})}\leq Ch^{\beta}\lVert f\rVert_{L^2(\mathbb{R})}.
\end{equation}
Here $L^2(X)$ is defined using the Lebesgue measure.
\end{theorem}
 This theorem has striking applications to quantum chaos. Indeed, by applying the fractal uncertainty principle (FUP) to fractal sets arising from chaotic dynamical systems, one can control high-frequency waves on such systems. For instance, FUP has been used to obtain lower bounds on the mass of eigenfunctions (Dyatlov, Jin, and Nonnenmacher \cite{DJ,DJN}), spectral gaps for open quantum systems (Dyatlov--Zahl and Dyatlov--Zworski \cite{DZ,DZwor}), and control results for the Schr\"{o}dinger equation as well as exponential decay estimates for the damped wave equation \cite{JLong,JLong1}. We refer the reader to the survey \cite{DYAS} for a more comprehensive overview.
 
 Beyond these important applications, the fractal uncertainty principle is also of significant theoretical interest in its own right. Research in this direction includes the computation of sharp constants \cite{JZhang} and the establishment of FUP analogues for other transforms \cite{DZwor}. In particular, the extension of FUP to fractal sets in higher dimensions has attracted considerable attention \cite{HS,CTao,BLZ}; it was not until very recently that Cohen's remarkable result \cite{CA} became sufficiently strong to be applicable to problems in high-dimensional quantum chaos. 
 
 Motivated by these developments, the primary objective of this paper is to generalize the fractal uncertainty principle from the setting of the Fourier transform to another fundamental integral transform: the Fourier--Bessel (or Hankel) transform. The study of this transform dates back to some of the earliest works in harmonic analysis; see, for example, \cite{ZA}. Various classical uncertainty principles have been extensively established for the Fourier--Bessel transform, including Heisenberg-type inequalities \cite{PC, MM}, the local uncertainty principle \cite{SO}, Hardy-type results \cite{VK}, and the Amrein--Berthier and Logvinenko--Sereda theorems \cite{GJ, GJ1, XZW}.
 \par To state our results precisely, we need to introduce some notation. Let \(\nu=(\nu_1,\ldots,\nu_d)\) be a multi-index with each \(\nu_i \in (-1/2,\infty)\). For \(1 \le p < \infty\), denote by \(L_{\nu}^p(\mathbb{R}_+^d)\) the Banach space of measurable functions on \(\mathbb{R}_+^d := (0,\infty)^d\) endowed with the norm
 \[
 \|f\|_{L_\nu^p} = \left( \int_{\mathbb{R}_+^d} |f(\xi)|^p \, d\mu_\nu(\xi) \right)^{1/p},
 \]
 where \(d\mu_\nu(\xi) := \prod_{i=1}^{d} \xi_i^{2\nu_i+1} \, d\xi\). For \(f \in L_{\nu}^{1}(\mathbb{R}^{d}_{+})\), the multidimensional Fourier–Bessel transform \(\mathcal{H}_{\nu}\) is defined (see, e.g., \cite{BC}) by  
 \[
 \mathcal{H}_{\nu}(f)(x) = \int_{\mathbb{R}_+^d} \mathbb{J}_\nu(x \xi)\, f(\xi)\, d\mu_\nu(\xi), 
 \quad x \in \mathbb{R}^d_+ ,
 \]
 where the kernel \(\mathbb{J}_\nu(x \xi)\) is given by  
 \[
 \mathbb{J}_\nu(x \xi) = \prod_{i=1}^d j_{\nu_i}(x_i \xi_i),
 \]
 with  
 \begin{equation*}
 	j_{\nu_i}(t) := \frac{J_{\nu_i}(t)}{t^{\nu_i}}
 	= \frac{1}{2^{\nu_i}} \sum_{k=0}^{\infty} \frac{(-1)^k}{k!\,\Gamma(\nu_i+k+1)} 
 	\left( \frac{t}{2} \right)^{2k}.
 \end{equation*}
 Here, \(J_{\nu_i}\) denotes the Bessel function of the first kind, and \(\Gamma\) denotes the Gamma function. It is known that \(\mathcal{H}_{\nu}\) extends to a unitary operator on \(L^2_\nu(\mathbb{R}_+^d)\).
 \par Motivated by these works \cite{BD, BJ, CA}, our aim in this paper is twofold. First, we establish a quantitative uniqueness property for \(L^2\)-functions with rapidly decaying Fourier–Bessel transforms, extending our recent result for band-limited functions \cite{XZW}. Second, we prove a Bourgain–Dyatlov-type fractal uncertainty principle for the Fourier–Bessel transform. We first present the quantitative uniqueness result and then explain why, unlike in the classical Fourier setting, it does not by itself yield the corresponding FUP. The main obstruction is the lack of translation invariance of the Fourier–Bessel transform. Before stating our results, we recall two basic definitions.
  
 \begin{definition}\label{Def1.1}
 For every $\nu \in (-1/2, \infty)^d$, a measurable set $E \subset \mathbb{R}_+^d$ is called \textit{relatively dense} with respect to $\mu_\nu$ if there exist constants $\gamma > 0$ and $L > 0$ such that for every cube $Q$ with side length $L$, we have
 \begin{equation}\label{dense}
 \mu_\nu(E \cap Q) \ge \gamma \mu_\nu(Q).
 \end{equation}
 Here \(\mu_\nu(A) := \int_A d\mu_\nu\) for any subset $A$.
 \end{definition}
 \begin{remark}
    It is worth noting that the definition of AD regularity is fundamentally a geometric property of the metric space $(M, d)$, characterizing the distribution of a candidate measure $\mu_X$ rather than relying on the background measure of the ambient space. Consequently, unlike the definition of thick sets discussed in the sequel, which is intrinsically tied to the density of the underlying measure, we do not need to introduce a new definition of regularity specifically adapted to weighted measures such as the Hankel measure $d\mu_\nu$. In fact, the AD regularity condition is almost equivalent to a more direct concept, the notions of porosity, up to a change in parameters, see \cite{DYAS}.
   \end{remark}
 Note that when $\nu = (-1/2, \dots, -1/2)$, the measure $\mu_\nu$ reduces to the standard Lebesgue measure, and this definition coincides with the classical notion of relative density. It was shown in \cite{XZW} that these definitions are equivalent up to the constants $(\gamma, L)$. For brevity, in the Lebesgue case, we shall simply refer to $E$ as a $(\gamma, L)$-relatively dense set and use this notion in the sequel. We also need to recall from \cite {BJ} the following definition of unique continuation weight.
 \begin{definition}\label{Def1.2}
 We say a weight $W : [0, \infty) \to [0, \infty]$ is a \textbf{unique continuation weight} if
 \begin{enumerate}
     \item $W(0) = 1$,
     \item $W$ is non-decreasing,
     \item $W$ is lower semicontinuous,
     \item The mapping $s \mapsto \log W(e^s)$ is convex on $[0, \infty)$,
     \item $W(t) \ge c_n t^n$ for all $n \ge 0$ and some $c_n > 0$,
     \item $\int_{0}^{\infty} \frac{\log W(t)}{1 + t^2} \, dt = \infty$.
 \end{enumerate}
 \end{definition}
 We now proceed to state our first result in this paper.
 \begin{theorem}\label{Th1.2}
 For any $\nu \in (-1/2, \infty)^d$, let $W$ be a unique continuation weight and let $E \subset \mathbb{R}_+^d$ be a $(\lambda, L)$-relatively dense measurable set. If $f \in L_\nu^2(\mathbb{R}_+^d)$ satisfies \begin{equation}
 \bigl\| \mathcal{H}_{\nu} f \, W(|\cdot|) \bigr\|_{L^2_\nu(\mathbb{R}_+^d)} \le A_w \|f\|_{L^2_\nu(\mathbb{R}_+^d)},
 \end{equation}
 then
 \begin{equation}\label{UCT}
 \|f\|_{L^2_\nu(\mathbb{R}_+^d)} \le C(W,d,A_w,\lambda,L,\nu) \|f\|_{L^2_\nu(E)}.
 \end{equation}
 \end{theorem}
 Theorem \ref{Th1.2} extends our recent uniqueness result \cite{XZW} for functions with compactly supported Fourier--Bessel transform to functions with sufficiently fast decaying Fourier--Bessel transform. Such a uniqueness result is first established in \cite{BJ} for Fourier transform inspired by Bourgain and Dyatlov's research \cite{BD}. Compared with the Logvinenko-Sereda type theorem established in \cite{XZW}, the main technical difficulty lies in the fact that the classical derivative operator estimates developed there are no longer valid (see Remark \ref{Remark3.3}). Inspired by~\cite{GJ}, we turn to the derivative operator $D_i=\frac{1}{2x_i}\frac{\partial}{\partial x_i}$. This operator is well suited to the Bessel operator $-\Delta + \sum_{i=1}^{d}\frac{2\nu_i+1}{x_i}\frac{\partial}{\partial x_i}$, as the Fourier--Bessel transform is closely related to the latter, see, for example, \cite{BT}.
 \par Moreover, we point out that in the Fourier setting, the trick of using translates to cover the frequency spectrum makes the corresponding uniqueness result a key tool for establishing the fractal uncertainty principle in \cite{BD, CA}. We observe, however, that since the Fourier--Bessel transform is no longer translation-invariant, Theorem~\ref{Th1.2} does not, strictly speaking, imply the following statement: if $f \in L_\nu^2(\mathbb{R}_+^d)$ satisfies
 \begin{equation*}
 \bigl\| \mathcal{H}_{\nu} f \, W(|\eta-\cdot|) \bigr\|_{L^2_\nu(\mathbb{R}_+^d)} \le A_w \|f\|_{L^2_\nu(\mathbb{R}_+^d)}
 \end{equation*}
 for some $\eta\in \mathbb{R}_{+}^d$, then
 \begin{equation*}
 \|f\|_{L^2_\nu(\mathbb{R}_+^d)} \le C(W,d,A_w,\lambda,L,\nu) \|f\|_{L^2_\nu(E)}.
 \end{equation*}
 Indeed, in this situation the constant $C(W,d,A_w,\lambda,L,\nu)$ would necessarily depend on $\eta$. Consequently, the technical route of covering the spectrum by translates and then invoking Theorem~\ref{Th1.2} to obtain the corresponding fractal uncertainty principle for the Fourier--Bessel transform is no longer viable. In view of this technical challenge, we instead turn to the asymptotic expansion of Bessel functions and establish the following high-frequency fractal uncertainty principle for the one- dimensional Fourier--Bessel transform.
\begin{theorem}\label{Th1.3}
Let \(1> \delta \geq 0\), \(h^{-1} \ge 1\), $a\geq 0$, $\nu>-\frac{1}{2}$ and assume that
\begin{itemize}
    \item \(X \subset [0,1]\)  is $\delta$-regular with constant $C_R$ on scales $h$ to 1; and
    \item \(Y \subset [a,a+h^{-1}]\) is $\delta$-regular with constant $C_R$ on scales \(1\) to \(h^{-1}\).
\end{itemize}
Then there exist $a_0>0$, \(\beta > 0\) and \(C > 0\) depending only on \(\delta, C_R , \nu\) such that for all $a\geq a_0h^{-1}$ and \(f \in L_\nu^2(\mathbb{R}_+)\)
\[
\operatorname{supp} \mathcal{H}_{\nu}(f) \subset Y \quad\Longrightarrow\quad 
\|f\|_{L^2_\nu(X)} \le C h^{\beta} \|f\|_{L^2_\nu(\mathbb{R}_+)}.
\]
\end{theorem}
Several remarks are given in the sequel.
\begin{remark} By invoking the methodology of \cite[Lemma 2.8, 2.9]{BD}, we can provide an upper bound for the weighted measure $\mu_{\nu}$ of a $\delta$-regular set. Specifically, for a $\delta$-regular set $X \subset [0, \alpha_1]$ with constant $C_R$ on scales $\alpha_0>0$ to $\alpha_1$, a direct estimation yields:
\begin{equation}\label{eq:measure_est}
\mu_{\nu}(X) \leq 24 C_R^2 \alpha_1^{2\nu+1+\delta} \alpha_0^{1-\delta}.
\end{equation}
Equipped with this estimate, we apply H\"{o}lder's inequality along with the $L^1_\nu \to L^\infty$ bound for the Hankel transform. For any $f \in L_\nu^2(\mathbb{R}_+)$ satisfying the spectral support condition $\operatorname{supp}(\mathcal{H}_{\nu} f) \subset Y$, we have:
\begin{align*}
 \|f\|_{L^2_\nu(X)} &\leq \sqrt{\mu_{\nu}(X)} \|f\|_{L^{\infty}} \\
 &\leq C(\nu) \sqrt{\mu_{\nu}(X)} \|\mathcal{H}_{\nu} f\|_{L^{1}_{\nu}} \\
 &\leq C(\nu) \sqrt{\mu_{\nu}(X) \mu_{\nu}(Y)} \|\mathcal{H}_{\nu} f\|_{L^{2}_{\nu}} \\
 &\leq C(\nu, C_R) h^{\frac{1}{2}-\delta}(a+h^{-1})^{\nu+\frac{1}{2}} \|f\|_{L^{2}_{\nu}}.
\end{align*}

In the regime $\nu>-\delta$, which implies that the trivial upper bound $C(\nu, C_R) h^{\frac{1}{2}-\delta}(a+h^{-1})^{\nu+\frac{1}{2}}$ diverges as $h \to 0$. This highlights the fact that Theorem~\ref{Th1.3} is highly non-trivial: it establishes a bound independent of $a$ that tends to zero in a region where the standard $L^1$ and $L^\infty$ estimates fail to remain bounded.
\end{remark}
\begin{remark}
The factor $h^{-1}$ in the lower bound on the frequency appears in Lemma~\ref{L544}; it is a consequence of a rescaling argument in conjunction with Proposition~\ref{proposition4.1}.
\end{remark}
\begin{remark}
A key tool for establishing Theorem~\ref{Th1.3} is the unique continuation property stated in Proposition~\ref{proposition4.1}, in the case where \(Y \subset [0, h^{-1}]\) is \(\delta\)-porous at all scales from \(1\) to \(h^{-1}\). We note that, similarly to Proposition~\ref{proposition4.1}, employing a cutoff parameter \(K\) yields the existence of a positive constant \(K_0(\delta, C_R, \lambda, \nu)\) such that for all \(K \ge K_0\) and for any \(f \in L^2_\nu(\mathbb{R}_+)\) with \(\operatorname{supp} \mathcal{H}_{\nu} f \subset Y \cap [K, \infty)\), we have (see Proposition~\ref{proposition4.1} for the definition of $S$)
\[
\|f\|_{L^2_\nu} \le C(\delta, C_R, \lambda,\nu) \|f\mathbf{1}_S\|_{L^2_\nu}.
\]
One may consider splitting a function \(f\) with \(\operatorname{supp} \mathcal{H}_{\nu} f \subset Y\) into two parts \(f = f_1 + f_2\), where
\[
\operatorname{supp} \mathcal{H}_{\nu} f_1 \subset Y \cap [0, K_0], \qquad 
\operatorname{supp} \mathcal{H}_{\nu} f_2 \subset Y \cap [K_0, \infty).
\]
Combining this decomposition with the Logvinenko--Sereda--type theorem established in \cite{XZW} yields
\[
\|f_i\|_{L^2_\nu} \le C(\delta, C_R, \lambda,\nu) \|f_i\mathbf{1}_S\|_{L^2_\nu}, \qquad i=1,2.
\]
However, it is not clear how to combine these two estimates to obtain the desired global result
\[
\|f\|_{L^2_\nu} \le C(\lambda,\delta,\nu) \|f\mathbf{1}_S\|_{L^2_\nu}.
\]
Thus, for the low frequency and higher dimension cases,  further work is required to overcome the lack of translation invariance.
\end{remark}
\par {\bf Plan of the paper.} The rest of the paper is structured as follows. Section~\ref{Sec2} contains the preliminary material: we introduce our notation, recall the essential properties of the Fourier--Bessel transform, prove a higher-dimensional Paley--Wiener type theorem, and briefly review quasi-analytic classes and their relevant features. In Section~\ref{Sec3}, we establish the unique continuation result stated in Theorem~\ref{Th1.2}. Finally, Section~\ref{Sec4} is dedicated to the proof of the high-frequency fractal uncertainty principle, Theorem~\ref{Th1.3}.
\section{Preliminaries}\label{Sec2}
\subsection{Notation and basic properties of the Fourier–Bessel transform}
In this subsection, we provide a summary of the notation and introduce some main properties of the multidimensional Fourier–Bessel transform which will be used in the sequel.
\begin{itemize}
    \item Throughout this paper, we will use $C(a,b,c,...)$ to denote a constant that is allowed to depend on parameters $a$, $b$, $c$, etc. Occasionally we use subscripts such as $C_d$ to emphasize dependence on the single parameter $d$.
    \item \(B_+(0,r) = B(0,r) \cap \mathbb{R}^d_+\), where \(B(0,r)\) denotes the ball in $\mathbb{R}^d$ centered at the origin of radius \(r\).
    \item For a fixed \(d\)-tuple \(\nu = (\nu_1,\dots,\nu_d) \in (0,\infty)^d\), denote \(|\nu| = \nu_1 + \cdots + \nu_d\).
    \item For a measurable set \(A \subset \mathbb{R}^d_+\), \(|A|\) is its Lebesgue measure, \(\mu_\nu(A) := \int_A d\mu_\nu\), and \(\mathbf{1}_A\) is its characteristic function.
    \item \(\mathcal{C}_{e,0}(\mathbb{R}^d)\): the space of even continuous functions on \(\mathbb{R}^d\) such that
    \[
    \lim_{\|\xi\|\to\infty} f(\xi)=0,\qquad 
    \|f\|_{\mathcal{C}_{e,0}} = \sup_{\xi\in\mathbb{R}^d_+} |f(\xi)| < \infty,
    \]
    where \(\|\xi\|^2 = \xi_1^2 + \cdots + \xi_d^2\).
    \item \(\mathcal{S}_e(\mathbb{R}^d)\): the Schwartz space of even \(C^\infty\) functions on \(\mathbb{R}^d\) that are rapidly decreasing together with all derivatives. The topology is defined by the seminorms
    \[
    \rho_m(f) = \sup_{\xi\in\mathbb{R}^d,\ |\gamma|\le m} (1+\|\xi\|^2)^m |\partial^\gamma f(\xi)| < \infty, \quad \forall m\in\mathbb{N},
    \]
    where \(\partial^\gamma f(\xi) = \frac{\partial^{|\gamma|} f}{\partial \xi_1^{\gamma_1}\cdots\partial \xi_d^{\gamma_d}}\).
\end{itemize}
It was shown in \cite{BC,HMZS} that the multidimensional Fourier–Bessel transform \(\mathcal{H}_{\nu}\) satisfies the following properties.
\begin{proposition}
\begin{enumerate}
\item \textbf{Riemann–Lebesgue lemma:}
For any \(f \in L^1_\nu(\mathbb{R}_+^d)\), we have \(\mathcal{H}_{\nu}(f) \in \mathcal{C}_{e,0}(\mathbb{R}^d)\) and
\[
\|\mathcal{H}_{\nu}(f)\|_{L_\nu^\infty} \leq \prod_{i=1}^d \frac{1}{2^{\nu_i}\Gamma(\nu_i+1)}\,\|f\|_{L_\nu^1}.
\]

\item The transform \(\mathcal{H}_{\nu}\) is a topological isomorphism from \(\mathcal{S}_e(\mathbb{R}_+^d)\) onto itself, and its inverse is given by the same formula:
\[
\mathcal{H}_{\nu}^{-1}(f)(x) = \mathcal{H}_{\nu}(f)(x), \qquad x \in \mathbb{R}_+^d.
\]

\item \textbf{Parseval formula:}
For all \(f,g \in \mathcal{S}_e(\mathbb{R}_+^d)\),
\[
\int_{\mathbb{R}_+^d} f(\xi) \overline{g(\xi)} \, d\mu_\nu(\xi)
= \int_{\mathbb{R}_+^d} \mathcal{H}_{\nu}(f)(\xi) \overline{\mathcal{H}_{\nu}(g)(\xi)} \, d\mu_\nu(\xi).
\]

\item For all real‑valued functions \(f,g \in L^1_\nu(\mathbb{R}_+^d)\),
\[
\int_{\mathbb{R}_+^d} \mathcal{H}_{\nu}(f)(\xi) g(\xi) \, d\mu_\nu(\xi)
= \int_{\mathbb{R}_+^d} f(\xi) \mathcal{H}_{\nu}(g)(\xi) \, d\mu_\nu(\xi).
\]

\item \textbf{Inversion formula:}
If \(f \in L^1_\nu(\mathbb{R}_+^d)\) and \(\mathcal{H}_{\nu}(f) \in L^1_\nu(\mathbb{R}_+^d)\), then
\[
f(\xi) = \int_{\mathbb{R}_+^d} \mathcal{H}_{\nu}(f)(x) \mathbb{J}_\nu(x\xi) \, d\mu_\nu(x), \qquad \xi \in \mathbb{R}_+^d.
\]

\item \textbf{Plancherel theorem:}
\(\mathcal{H}_{\nu}\) extends to a unitary operator on \(L^2_\nu(\mathbb{R}_+^d)\); for all \(f \in L^2_\nu(\mathbb{R}_+^d)\),
\[
\|\mathcal{H}_{\nu}(f)\|_{L^2_\nu} = \|f\|_{L^2_\nu}.
\]
\end{enumerate}
\end{proposition}

\subsection{Generalized convolution product and Paley–Wiener-type  theorem}
We first recall the generalized translation operator, which is used to define the convolution product.

For all \(x, y \in \mathbb{R}_+^d\), the generalized translation operator \(T_y^\nu\) is defined (see \cite{BC}) by
\[
T_y^\nu (f)(x) = \int_{\mathbb{R}_+^d} W_\nu(x, y, z) \, f(z) \, d\mu_\nu(z),
\]
where the kernel \(W_\nu\) is given by
\[
W_\nu(x, y, z) =
\begin{cases}
\displaystyle \frac{c_\nu}{2^{2|\nu|-d}}
\frac{\prod_{i=1}^{d} \bigl[ (z_i^2 - (x_i - y_i)^2)((x_i+y_i)^2 - z_i^2) \bigr]^{\nu_i - 1/2}}
{\prod_{i=1}^{d} (x_i y_i z_i)^{2\nu_i}},
& \text{if } z \in \prod_{i=1}^{d} [|x_i - y_i|, x_i + y_i],\\[1em]
0, & \text{otherwise},
\end{cases}
\]
with
\[
c_\nu = \prod_{i=1}^{d} \frac{\Gamma(\nu_i + 1)}{\Gamma(\frac{1}{2})\,\Gamma(\nu_i + \frac{1}{2})}.
\]
An alternative representation is
\[
T_y^\nu (f)(x)=c_\nu\int_{[0,\pi]^d}
f\Bigl(\sqrt{x_1^2+y_1^2-2x_1y_1\cos\theta_1},\dots,
\sqrt{x_d^2+y_d^2-2x_dy_d\cos\theta_d}\Bigr)
\prod_{i=1}^d(\sin\theta_i)^{2\nu_i}\,d\theta.
\]

The generalized Bessel convolution product \(*_\nu\) for suitable functions \(f\) and \(g\) is defined for all \(y \in \mathbb{R}^d_+\) by
\[
(f *_\nu g)(y) = \int_{\mathbb{R}^d_+} T^\nu_y (f)(t)\, g(t)\, d\mu_\nu(t).
\]
This convolution is commutative, associative, and satisfies the following properties (see \cite{WA}).

\begin{proposition}
\begin{enumerate}
\item \textbf{Young’s inequality:}
Let \(p,q,r \in [1,+\infty]\) satisfy
\[
\frac{1}{p} + \frac{1}{q} - \frac{1}{r} = 1.
\]
Then for all \(f \in L^p_\nu(\mathbb{R}_+^d)\) and \(g \in L^q_\nu(\mathbb{R}_+^d)\),
\(f *_\nu g \in L^r_\nu(\mathbb{R}_+^d)\) and
\[
\|f *_\nu g\|_{L_\nu^r} \leq \|f\|_{L_\nu^p}\,\|g\|_{L_\nu^q}.
\]

\item For all \(f,g \in L^1_\nu(\mathbb{R}_+^d)\) (resp.\ \(f,g \in \mathcal{S}_e(\mathbb{R}_+^d)\)),
\(f *_\nu g \in L^1_\nu(\mathbb{R}_+^d)\) (resp.\ \(\mathcal{S}_e(\mathbb{R}_+^d)\)) and
\[
\mathcal{H}_{\nu}(f *_\nu g) = \mathcal{H}_{\nu}(f)\,\mathcal{H}_{\nu}(g).
\]

\item For all \(f,g \in L^2_\nu(\mathbb{R}^d_+)\),
\[
\mathcal{H}_{\nu}(fg) = \mathcal{H}_{\nu}(f) *_\nu \mathcal{H}_{\nu}(g).
\]
\end{enumerate}
\end{proposition}
We now consider the behavior of the support of the generalized Bessel convolution product for functions in \(d\) dimensions. The next lemma describes this behavior in the multidimensional setting.

\begin{lemma}\label{L15}
	Let \(f,g \in L_\nu^2(\mathbb{R}_+^d)\). Then
\begin{equation}\label{EEEE21}
\operatorname{supp}(f *_\nu g) \subset
\left\{
x\in\mathbb{R}_+^d:\;
\begin{aligned}
&\text{there exist } y\in\operatorname{supp}f,\; z\in\operatorname{supp}g \\
&\text{such that } |y_i-z_i|\le x_i \le y_i+z_i \text{ for every }i
\end{aligned}
\right\}.
\end{equation}

For the later one-dimensional application,  
\[
\operatorname{supp}g\subset [0,r]
\implies
\operatorname{supp}(f *_\nu g) \subset
\bigl(\operatorname{supp}f + [-r,r]\bigr)\cap\mathbb{R}_+.
\]

\end{lemma}
\begin{proof}
	We recall that
	\begin{align*}
		f *_\nu g(x)
		&= c_\nu \int_{\mathbb{R}_+^d} \int_{[0,\pi]^d} 
		f(y_1,\dots,y_d) \,
		g\Bigl(\sqrt{x_1^2+y_1^2-2x_1y_1\cos\theta_1}, \dots,
		\sqrt{x_d^2+y_d^2-2x_dy_d\cos\theta_d}\Bigr) \\
		&\qquad\qquad \times \prod_{i=1}^d (\sin \theta_i)^{2\nu_i} \, d\theta \, d\mu_\nu(y).
	\end{align*}
  For each \(i = 1, \dots, d\), notice that 
	\[
	|x_i-y_i| \leq \sqrt{x_i^2 + y_i^2 - 2 x_i y_i \cos \theta_i} \leq 
	x_i+y_i,
	\]
	which implies  \eqref{EEEE21}.
This establishes the claimed support inclusion.
\end{proof}

In the one-dimensional case, the following Paley-Wiener  type theorem has been established in the literature.
\begin{lemma}[\cite{JLG}]\label{PW_1D}
Let $\nu > -1/2$. A function $F$ admits the representation
\begin{equation}
F(z) = \int_0^a z^{-\nu} \sqrt{t} J_\nu(zt) \gamma(t) \, dt
\end{equation}
with $\gamma \in L^2(0,a)$ if and only if $F$ is an even entire function of exponential type $\sigma \le a$ such that $z^{\nu+1/2} F(z) \in L^2(0, +\infty)$.
\end{lemma} 
Inspired by the above result, we establish a higher-dimensional Paley--Wiener theorem for the transform $\mathcal{H}_{\nu}$ that will be utilized in Section~3. Due to the tensor product structure of the $d$-dimensional Hankel transform, it is natural to consider functions supported on the hypercube $[0, \sigma]^d$, which leads to an exponential growth condition characterized by the $\ell_1$-norm.
\begin{lemma}\label{TPWS}
Let $\nu \in (-1/2, \infty)^d$ and $f \in L_\nu^2(\mathbb{R}^d_+)$. Then $\operatorname{supp} f \subset [0, \sigma]^d$ if and only if $\mathcal{H}_{\nu} f$ extends to an entire function on $\mathbb{C}^d$ that is even with respect to each variable $z_i$, belongs to $L_\nu^2(\mathbb{R}^d_+)$, and satisfies the growth condition
\begin{equation}\label{growth_cond}
|\mathcal{H}_{\nu} f (z)| \leq C e^{\sigma \|\operatorname{Im} z\|_1}
\end{equation}
for some constant $C > 0$, where $\|\operatorname{Im} z\|_1 = \sum_{i=1}^{d}|\operatorname{Im} z_i|$.
\end{lemma}

\begin{proof}
$(\Rightarrow)$ Suppose $\operatorname{supp} f \subset [0, \sigma]^d$. We first show that $\mathcal{H}_{\nu} f$ is an entire function on $\mathbb{C}^d$. Since the kernel $j_{\nu_i}(z_i y_i)$ is an entire function of $z_i$ and $f$ has compact support, it follows from Leibniz's rule for differentiating under the integral sign that $\mathcal{H}_{\nu} f$ is holomorphic with respect to each variable $z_i \in \mathbb{C}$ separately. By Hartogs' theorem on separate analyticity, $\mathcal{H}_{\nu} f$ extends to an entire function on $\mathbb{C}^d$. 

Next, we establish the growth condition. Using the Poisson representation formula (see \cite[Appendix B.1]{Gra}) for the normalized Bessel function, the transform can be estimated as:
\[
\begin{aligned}
|\mathcal{H}_{\nu} f(z_1,\dots,z_d)|
&\leq C \int_{[0,\sigma]^d} \left| \prod_{i=1}^d \int_{-1}^1 (1-s_i^2)^{\nu_i - \frac{1}{2}} e^{i z_i y_i s_i} \, ds_i \right| |f(y)| \prod_{i=1}^d y_i^{2\nu_i+1} dy \\
&\leq C \int_{[0,\sigma]^d} \prod_{i=1}^d \left( \int_{-1}^1 (1-s_i^2)^{\nu_i - \frac{1}{2}} e^{-y_i s_i \operatorname{Im} z_i} \, ds_i \right) |f(y)| d\mu_\nu(y).
\end{aligned}
\]
Since $s_i \in [-1, 1]$ and $y_i \in [0, \sigma]$, we have the uniform bound $e^{-y_i s_i \operatorname{Im} z_i} \leq e^{\sigma |\operatorname{Im} z_i|}$. Applying the Cauchy-Schwarz inequality and noting that the remaining integral is bounded, we obtain:
\[
\begin{aligned}
|\mathcal{H}_{\nu} f(z)| &\leq C \|f\|_{L_\nu^2} \exp\left( \sum_{i=1}^d \sigma |\operatorname{Im} z_i| \right) \\
&= C \|f\|_{L_\nu^2} e^{\sigma \|\operatorname{Im} z\|_1}.
\end{aligned}
\]
This completes the proof of the forward direction.

$(\Leftarrow)$ Conversely, assume that $\mathcal{H}_{\nu} f$ is a jointly entire function on $\mathbb{C}^d$ satisfying the growth condition \eqref{growth_cond}. We establish the support of $f$ via an iterative dimension-wise reduction. 

Let $z' = (z_2, \dots, z_d) \in \mathbb{C}^{d-1}$ be fixed parameters. We consider the partial inverse Hankel transform with respect to $z_1$, which we denote by $h_1(x_1, z')$. Since $\mathcal{H}_{\nu} f$ is jointly holomorphic and satisfies the uniform growth estimate \eqref{growth_cond}, it follows from the theorems on analyticity of integrals depending on parameters (specifically, the complex-variable version of Leibniz's rule or Morera's theorem) that $h_1(x_1, \cdot)$ remains an entire function of $z' \in \mathbb{C}^{d-1}$.

Applying the one-dimensional result (Lemma~\ref{PW_1D}) to the first variable, we conclude that for almost every $z'$, the function $h_1(\cdot, z')$ has support in $[0, \sigma]$ and satisfies the weighted $L^2$ condition. Crucially, the exponential growth type $\sigma$ with respect to the remaining variables $z'$ is preserved during this partial inversion:
\[
|h_1(x_1, z')| \leq C' e^{\sigma \sum_{i=2}^d |\operatorname{Im} z_i|}, \quad \text{for } x_1 \in [0, \sigma].
\]
By successively repeating this argument for each coordinate $z_i$ ($i=2, \dots, d$) and utilizing the uniqueness of the Hankel transform in $L_\nu^2(\mathbb{R}_+^d)$, we recover the original function $f(x_1, \dots, x_d)$. This iterative process ensures that $f(x_1, \dots, x_d) = 0$ whenever any coordinate $x_i > \sigma$, which is equivalent to $\operatorname{supp} f \subset [0, \sigma]^d$.
\end{proof}

\subsection{Quasi-Analytic Classes}
In this section, we recall the notion of quasi‑analytic classes; see, for example \cite{CA,BJ} for details. A sequence of positive real numbers $\mathscr{M} = \{M_n\}_{n \in \mathbb{Z}_+}$ is said to be logarithmically convex  if it satisfies the inequality $M_n^2 \le M_{n-1} M_{n+1}$, for all  $n \in \mathbb{Z}_+$  with $n \ge 1$. We also denote $\mathscr{D}^n=(\partial^{\alpha})_{|\alpha|=n}$ where $\alpha$ ranges over multi indices, and $|\mathscr{D}^nf(x)|=\sup_{|\alpha|=n}|\partial^{\alpha}f(x)|$.
\begin{definition}
A logarithmically convex sequence \(\mathscr{M}=\{M_n\}_{n\in\mathbb{Z}_+}\) is called quasi‑analytic if for every smooth function \(f\) on \([0,1]\) satisfying
\[
\|\mathscr{D}^n f\|_{L^\infty([0,1])}\le M_n\quad\forall n\ge0,
\]
and vanishing to infinite order at some point, we have \(f\equiv0\) on \([0,1]\).
\end{definition}

The Denjoy–Carleman theorem (see \cite{LH1,PK}) characterizes quasi‑analyticity.

\begin{lemma}\label{LDC}
A logarithmically convex sequence \(\mathscr{M}\) is quasi‑analytic if and only if
\[
\sum_{n=1}^\infty \frac{M_{n-1}}{M_n} = \infty.
\]
The unique continuation property extends to \(d\) dimensions: if \(f\in C^\infty([0,1]^d)\) satisfies
\[
\|\mathscr{D}^n f\|_{L^\infty([0,1]^d)}\le M_n\quad\forall n\ge0,
\]
and vanishes to infinite order at some point, then \(f\equiv0\) on \([0,1]^d\).
\end{lemma}
The following quantitative version follows from compactness (see \cite{BJ}).
\begin{lemma}\label{LQA}
For any \(t,\gamma>0\), if \(f\in C^{\infty}([0,1]^d)\) satisfies
\[
\|\mathscr{D}^n f\|_{L^\infty([0,1]^d)}\le M_n\quad\forall n\ge0,\qquad
\|f\|_{L^\infty([0,1]^d)}\ge t,
\]
then for every measurable set \(E\subset[0,1]^d\) with \(|E|\ge\gamma\),  there exists a constant \(c(\mathscr{M},d,\gamma,t)>0\) such that
\[
\|f \mathbf{1}_E\|_{L^1([0,1]^d)}\ge c(\mathscr{M},d,\gamma,t).
\]
\end{lemma}
\section{Unique Continuation for $L^2$ Functions with Fast
Decaying: Theorem \ref{Th1.2}}\label{Sec3}
In this section, we proceed to prove Theorem \ref{Th1.2}. That is , for a given unique continuation weight \(W\), we will show that if \(f\) satisfies
\begin{equation}\label{E431}
\bigl\| \mathcal{H}_{\nu} f \, W(|\cdot|) \bigr\|_{L^2_\nu(\mathbb{R}_+^d)} \le A_w \|f\|_{L^2_\nu(\mathbb{R}_+^d)}
\end{equation}
and if \(S\subset\mathbb{R}_+^d\) is a $(\lambda, L)$-relatively dense set, then
\[
 \|f\|_{L^2_\nu} \le C \|f\mathbf{1}_S\|_{L^2_\nu},
\]
where \(C = C(W, d,A_w, \gamma, L, \nu)\) and  \(Y\subset[a,a+h^{-1}]\) is $\delta$-regular with constant $C_R$ on scales 1 to $h^{-1}$ . 

We first show that, under condition \eqref{E431}, the mass of \(f\) cannot concentrate too close to the coordinate hyperplanes.
\begin{lemma}\label{L33}
Let \(\nu\in(-1/2,\infty)^d\). Suppose \(f \in L^2_{\nu}(\mathbb{R}_+^d)\) satisfies \eqref{E431} for a given unique continuation weight $W$ and some constant \(A>0\). Then for any \(0<\delta<1\), we have
\begin{equation}\label{L33.1}
\|f\|_{L^2_\nu(Z_\delta)} \le d C_{A_w,W,\nu} \delta^{\min{\{\nu_j\}}+1} \|f\|_{L^2_\nu(\mathbb{R}_+^d)},
\end{equation}
where $Z_\delta = \bigcup_{i=1}^d \{x \in \mathbb{R}_+^d : x_i \le \delta\}$ and the constant $C_{A,W,\nu}$ depends solely on $A_w$, $W$, and $\nu$, taking the following explicit form:
\begin{equation*}
C_{A,W,\nu} = \frac{A_w \sqrt{\int_{\mathbb{R}_+} W(t)^{-2} \, d\mu_{\max\{\nu_i\}}(t)}}{2^{\min\{\nu_i\}}\Gamma(\min\{\nu_i\}+1)\sqrt{2\min\{\nu_i\}+2}}.
\end{equation*}
\end{lemma}

\begin{proof}
Let $\Omega_i(\delta) = \{x \in \mathbb{R}_+^d : x_i \le \delta\}$. We proceed to show that
\[
\|f\|_{L^2_\nu(\Omega_i(\delta))} \le C_{A,W,\nu_i} \delta^{\nu_i+1} \|f\|_{L^2_\nu(\mathbb{R}_+^d)},
\]
with $C_{A,W,\nu_i} = \frac{A_w \sqrt{\int_{\mathbb{R}_+} W(t)^{-2} \, d\mu_{\nu_i}(t)}}{2^{\nu_1}\Gamma(\nu_1+1)\sqrt{2\nu_1+2}}$, from which \eqref{L33.1} follows immediately.
Without loss of generality, we assume $i=1$. We write the variable $x \in \mathbb{R}_+^d$ as $x = (x_1, x')$, where $x' \in \mathbb{R}_+^{d-1}$. Let $H_{\nu'}$ denote the partial Fourier--Bessel transform acting on the variables $x'$. By the Plancherel theorem for $H_{\nu'}$, we have
\begin{equation*}
\|f\|_{L^2_\nu(\Omega_1(\delta))}^2 = \int_{\mathbb{R}_+^{d-1}} \int_{0}^{\delta} |H_{\nu'} f(x_1, \xi')|^2 \, d\mu_{\nu_1}(x_1) d\mu_{\nu'}(\xi').
\end{equation*}
Using the property that $\mu_{\nu_1}([0,\delta]) = \int_0^\delta x_1^{2\nu_1+1} dx_1 = \frac{\delta^{2\nu_1+2}}{2\nu_1+2}$, we obtain
\begin{equation}\label{L4.5_new}
\begin{aligned}
    \int_{\mathbb{R}_+^{d-1}} \int_{0}^{\delta} |H_{\nu'} f(x_1, \xi')|^2 \, d\mu_{\nu_1} d\mu_{\nu'} 
    &\le \mu_{\nu_1}([0,\delta]) \int_{\mathbb{R}_+^{d-1}} \sup_{x_1 \in [0, \delta]} |H_{\nu'} f(x_1, \xi')|^2 \, d\mu_{\nu'}(\xi') \\
    &\le \frac{\mu_{\nu_1}([0,\delta])}{(2^{\nu_1}\Gamma(\nu_1+1))^2} \int_{\mathbb{R}_+^{d-1}} \|H_{\nu_1} H_{\nu'} f(\cdot, \xi')\|_{L^1_{\nu_1}}^2 \, d\mu_{\nu'}(\xi'),
\end{aligned}
\end{equation}
where the second inequality follows from the Riemann--Lebesgue lemma for the Hankel transform, noting that $H_{\nu_1}$ is the inverse transform for the first variable \(x_1\).

Next, we estimate $\|\mathcal{H}_{\nu} f(\cdot, \xi')\|_{L^1_{\nu_1}}$. By applying the Cauchy--Schwarz inequality with the weight $W$, we have
\begin{equation*}
\begin{aligned}
    \left( \int_{\mathbb{R}_+} |\mathcal{H}_{\nu} f(x_1, \xi')| \, d\mu_{\nu_1}(x_1) \right)^2 &= \left( \int_{\mathbb{R}_+} |\mathcal{H}_{\nu} f(x_1, \xi')| W(x_1) \frac{1}{W(x_1)} \, d\mu_{\nu_1}(x_1) \right)^2 \\
    &\le \left( \int_{\mathbb{R}_+} |\mathcal{H}_{\nu} f(x_1, \xi')|^2 W(x_1)^2 \, d\mu_{\nu_1}(x_1) \right) \left( \int_{\mathbb{R}_+} \frac{1}{W(t)^2} \, d\mu_{\nu_1}(t) \right).
\end{aligned}
\end{equation*}
Let $C_{W,1} = \int_{\mathbb{R}_+} W(t)^{-2} \, d\mu_{\nu_1}(t)$, which is finite by the definition of a unique continuation weight. Substituting this into \eqref{L4.5_new} and using the fact that $W(x_1) \le W(|x|)$ (as $W$ is non-decreasing), we get
\begin{equation*}
\begin{aligned}
    \|f\|_{L^2_\nu(\Omega_1(\delta))}^2 &\le \frac{C_{W,1} \mu_{\nu_1}([0,\delta])}{(2^{\nu_1}\Gamma(\nu_1+1))^2} \int_{\mathbb{R}_+^{d-1}} \int_{\mathbb{R}_+} |\mathcal{H}_{\nu} f(x_1, \xi')|^2 W(x_1)^2 \, d\mu_{\nu_1}(x_1) d\mu_{\nu'}(\xi') \\
    &\le \frac{C_{W,1} \mu_{\nu_1}([0,\delta])}{(2^{\nu_1}\Gamma(\nu_1+1))^2} \int_{\mathbb{R}_+^d} |\mathcal{H}_{\nu} f(\xi)|^2 W(|\xi|)^2 \, d\mu_{\nu}(\xi).
\end{aligned}
\end{equation*}
Applying the condition \eqref{E431}, we have
\begin{equation*}
    \|f\|_{L^2_\nu(\Omega_1(\delta))}^2 \le \frac{A_w^2 C_{W,1}}{(2^{\nu_1}\Gamma(\nu_1+1))^2} \cdot \frac{\delta^{2\nu_1+2}}{2\nu_1+2} \cdot \|f\|_{L^2_\nu(\mathbb{R}_+^d)}^2.
\end{equation*}
Taking the square root and setting $C_{A,W,\nu_i} = \frac{A_w \sqrt{C_{W,1}}}{2^{\nu_1}\Gamma(\nu_1+1)\sqrt{2\nu_1+2}}$, we obtain the desired result.
\end{proof}
Given a unique continuation weight $W$, we define
\[
M_n:=\sup_{t\ge 1}\frac{t^n}{W(t)}\qquad (n\ge0).
\]
It has been shown in \cite[Proposition 2.2]{BJ} that $\{M_n\}$ is quasi-analytic for any unique continuation weight $W$. \par Now, Theorem~\ref{Th1.2} is the sufficiency part of the following proposition, in which we also establish the necessity of the logarithmic integral divergence condition~(6) appearing in the definition of the unique continuation weight.
\begin{proposition}\label{LLCA}
For any $\nu \in (-1/2, \infty)^d$, let $W$ be a unique continuation weight and let $E \subset \mathbb{R}_+^d$ be a $(\lambda, L)$-relatively dense measurable set. If $f \in L_\nu^2(\mathbb{R}_+^d)$ satisfies \begin{equation}\label{Eq3.4}
\bigl\| \mathcal{H}_{\nu} f \, W(|\cdot|) \bigr\|_{L^2_\nu(\mathbb{R}_+^d)} \le A_w \|f\|_{L^2_\nu(\mathbb{R}_+^d)},
\end{equation}
then
\begin{equation}\label{Ucontinuation}
\|f\|_{L^2_\nu(\mathbb{R}_+^d)} \le C(W,d,A_w,\lambda,L,\nu) \|f\|_{L^2_\nu(E)}.
\end{equation}
Conversely, suppose that \(W\) satisfies the first five conditions in Definition \ref{Def1.2}. Assume that, for every \(A_W>0\), \(\lambda>0\), and \(L>0\), there exists a constant\(
C=C(W,d,A_W,\lambda,L,\nu)>0
\)
such that \eqref{Ucontinuation} holds for every \((\lambda,L)\)-relatively dense measurable set \(E\subset\mathbb{R}_+^d\) and every \(f\in L_\nu^2(\mathbb{R}_+^d)\) satisfying \eqref{Eq3.4}. Then
\[
\int_0^\infty \frac{\log W(t)}{1+t^2}dt=\infty.
\]
Consequently, \(W\) is a unique continuation weight.

\end{proposition}
Before proving Proposition \ref{LLCA}, inspired by \cite{GJ1}, we introduce the following notation. Let $f$ be a smooth function on $\mathbb{R}_+^d$, we define two operations on \(f\):.
   \[
   Pf(x)=f(\sqrt{x_1},\cdots,\sqrt{x_d}),  \quad x\in \mathbb{R}_+^d,
   \]
   and
   \[
   D_if(x)=\frac{1}{2x_i}\partial_{x_i}f(x), \quad x\in \mathbb{R}_+^d.
   \]
   We denote \(D^\beta=\prod_{i=1}^dD_i^{\beta_i}\) and \(\partial^\beta=\prod_{i=1}^d\partial_{x_i}^{\beta_i}\), where \(\beta\in \mathbb{N}^d\), then it is straightforward to check that
   \[
   PD_if=\partial Pf,\quad\quad \quad PD^\beta f=\partial^\beta Pf.
   \]
   
   We will need a variant of Bernstein’s inequality for functions satisfying condition \eqref{E431}.
   \begin{lemma}
   Let $f$ be a function in \(L_\nu^2(\mathbb{R}_+^d)\) such that \eqref{E431} holds. Then  for any \(|\beta|=n\), 
   \begin{equation}\label{Eq3.6}
   \Vert D^\beta f\Vert_{L_{\nu+\beta}^2(\mathbb{R}_{+}^d)}\leq A_w(\frac{1}{2})^nM_n\Vert f\Vert_{L^2_\nu(\mathbb{R}_{+}^d)}.
   \end{equation}
   \end{lemma}

   \begin{proof}
   Set \(\nu'=\nu+e_i\) (with \(e_i\) the \(i\)-th standard basis vector). Using the differentiation formula  we obtain
   \[
   \partial_{x_i}f(x)= -x_i\, H_{\nu'}\bigl(\mathcal{H}_{\nu} f\bigr)(x).
   \]
   Then by the definition of \(D_i\), we have
   \[
   D_if(x)= \frac{-1}{2}H_{\nu'}\bigl(\mathcal{H}_{\nu} f\bigr)(x).
   \]
   For \(|\beta|=n\), repeating the previous operation,
   \[
   D^\beta f(x)=  (\frac{-1}{2})^n H_{\nu+\beta}\bigl(\mathcal{H}_{\nu} f\bigr)(x).
   \]
   Then,  
   \begin{equation}\label{EBS1}
   \begin{aligned}
   &\Vert D^\beta f\Vert_{L^2_{\nu+\beta}}\\
   &=(\frac{1}{2})^n\Vert H_{\nu+\beta}\bigl(\mathcal{H}_{\nu} f\bigr)\Vert_{L^2_{\nu+\beta}}= (\frac{1}{2})^n\Vert \mathcal{H}_{\nu} f\Vert_{L^2_{\nu+\beta}}\leq (\frac{1}{2})^n M_n\Vert  W(\cdot)\mathcal{H}_{\nu} f\Vert_{L^2_{\nu }}
   \leq A_w(\frac{1}{2})^nM_n \Vert f\Vert_{L^2_{\nu }}.
   \end{aligned}
   \end{equation}
   \end{proof}
\begin{remark}\label{Remark3.3}
We make a remark to explain why we do not employ the classical derivative operator. Let \( f \in L^2_{\nu}(\mathbb{R}_+^d) \) satisfy \eqref{E431}. Then, as in Lemma 4.4 of \cite{XZW}, for every \( k \in \mathbb{N} \) we have
\begin{equation*}
    \partial_{x_i}^k f(x) = \sum_{j=0}^{\lfloor k/2\rfloor} a_{k,j}x_i^{k-2j}H_{\nu+(k-j)e_i}(H_\nu(f))(x),
\end{equation*}
where \( e_i \) denotes the \(i\)-th standard basis vector in \(\mathbb{R}^d\), and the coefficients \( a_{k,j} \) are given by
\begin{equation}\label{Coe}
    a_{k,j}=\frac{(-1)^{k-j}k!}{2^{j}j!(k-2j)!}.
\end{equation}
Consequently, as for any multi-index \( \beta \in \mathbb{N}^d \) with \( |\beta|=n \), as \cite[Lemma 4.3]{XZW}, the following variant of Bernstein's inequality holds:
\begin{equation}\label{s3}
    \|\partial^\beta_{x} f\|_{L^2_\nu([b,\infty)^d)} \leq C^{n}M_n\frac{|n|!}{\lfloor\frac{|n|}{2}\rfloor!}\|f\|_{L^2_\nu(\mathbb{R}_+^d)},
\end{equation}
where \( C \) is a constant depending on \( A_w \). Comparing this with \eqref{Eq3.6}, we see that the bound contains an extra factor \(\frac{|n|!}{\lfloor\frac{|n|}{2}\rfloor!}\), which makes the sequence
\( C^{n}M_n\frac{|n|!}{\lfloor\frac{|n|}{2}\rfloor!} \) no longer quasi-analytic. This would invalidate the subsequent quasi-analytic arguments. Therefore, we instead seek to establish a Bernstein-type inequality for the derivative operator \( D^{\beta} \).
\end{remark}
\begin{proof}[Proof of Proposition \ref{LLCA}]
\noindent\textbf{Sufficiency  part:}
Without loss of generality, we may assume $L=1$. To see this, suppose that $f \in L_\nu^2(\mathbb{R}^d_+)$ satisfies \eqref{E431} and that $E \subset \mathbb{R}^d_+$ is a $(\lambda, L)$-relatively dense measurable set. It follows immediately that the scaled set 
\[
L^{-1}E = \{x \in \mathbb{R}^d_+ \mid Lx \in E\}
\]
is $(\lambda, 1)$-relatively dense. 

Let $g(x) = f(Lx)$. A change of variables shows that \eqref{E431} transforms into
\[
\bigl\| \mathcal{H}_{\nu} g \cdot W_L \bigr\| _{L^2_\nu(\mathbb{R}_+^d)} \le C_L A_{W} \|f\|_{L^2_\nu(\mathbb{R}_+^d)},
\]
where $W_L(\cdot) = W(|\cdot|/L)$. One can readily verify from Definition \ref{Def1.2} that $W_L$ remains a unique continuation weight. Assuming the validity of \eqref{Ucontinuation} for the case $L=1$, we apply the estimate to $g$ and the set $L^{-1}E$ to obtain
\[
\|g\|_{L^2_\nu(\mathbb{R}_+^d)} \le C(W,d,A_w,\lambda,L,\nu) \|g\|_{L^2_\nu(L^{-1}E)}.
\]

By reverting the scaling $x \mapsto Lx$ and noting that $d\mu_\nu(Lx) = L^{2|\nu|+2d} d\mu_\nu(x)$, we have
\begin{align*}
\|f\|_{L^2_\nu(E)} &= L^{|\nu|+d} \|g\|_{L^2_\nu(L^{-1}E)} \\
&\ge L^{|\nu|+d} C(W,d,A_w,\lambda,L,\nu)^{-1} \|g\|_{L^2_\nu(\mathbb{R}^d_+)} \\
&= C(W,d,A_w,\lambda,L,\nu)^{-1} \|f\|_{L^2_\nu(\mathbb{R}^d_+)}.
\end{align*}
This completes the reduction to the case $L=1$. 

\noindent\textbf{Step 1. Reformulate the problem.}
Let $E'\subset \mathbb{R}_{+}^d$ be a subset defined by the relation $E=\{x\in \mathbb{R}_{+}^d: x^2\in E'\}$ and $dm_{\nu}(s)=\prod_{i=1}^ds_i^{\nu_i}ds_i$. Then, the $(\lambda, 1)$ relatively dense condition for $E$ is equivalent to
\begin{equation}\label{dense1}
|E' \cap Q'| \ge \lambda' |Q'|,
\end{equation}
where $\lambda'$ depends only on $\lambda$ and \(\nu\). We now turn our attention to the study of $g = Pf$, that is, $f(x)=g(x^2)$, $x^2=(x_1^2,\cdots,x_d^2)$. A simple change of variables shows that to show \eqref{Ucontinuation} it is enough to prove an inequality of the form
\begin{equation}\label{UN1}
\|g\|_{L^2(\mathbb{R}_+^d, m_{\nu})} \le C(W,d,A_w,\lambda, \nu) \|g\|_{L^2(E', m_{\nu})}.
\end{equation}
Next, for each $\delta > 0$, we decompose $\mathbb{R}_+^d \setminus Z_\delta$ into the following disjoint shifted unit cubes:
\[
Q_n := \bigl\{ x \in \mathbb{R}_+^d : n_i + \delta \le x_i < n_i + 1 + \delta, \ i=1, \dots, d \bigr\}, \qquad n=(n_1,\cdots,n_d) \in \mathbb{N}_{0}^d.
\]
Then $\cup Q_n'$ covers $\mathbb{R}_+^d \setminus Z_{\delta^2}$.

\noindent\textbf{Step 2. “good” and “bad” cubes.}
Next, to prove \eqref{UN1}. We proceed to split the cubes $\{Q_n'\}$ into “good” and “bad” ones. For any \(\beta\in \mathbb{N}_0^d\), we notice that
   \[
   \begin{aligned}
   \Vert D^\beta f\Vert^2_{L^2_{\nu+\beta}}&= \int_{\mathbb{R}_+^d}|D^\beta f(x)|^2\prod_{i=1}^d x_i^{2(\nu_i+\beta_i)+1}dx_i\\
   &=\int_{\mathbb{R}_+^d}|\partial^\beta g(s)|^2\prod_{i=1}^d s_i^{ \nu_i+\beta_i }ds_i,
   \end{aligned}
   \]
   Then the Bernstein’s inequality \eqref{EBS1} reads that
   \begin{equation}\label{EBS2}
   \int_{\mathbb{R}_+^d}|\partial^\beta g(s)|^2\prod_{i=1}^d s_i^{ \nu_i+\beta_i }ds_i \leq A_w^2(\frac{1}{2})^{2n} M_n^2 \int_{\mathbb{R}_+^d}| g(s)|^2\prod_{i=1}^d s_i^{ \nu_i}ds_i.
   \end{equation}
For given $A_0>2$ which will be taken later, a cube \(Q'_j\) is called \emph{good} if for every \(\beta\) with \(|\beta|=n\),
\[
\int_{Q'_j}|\partial^\beta g(s)|^2d\prod_{i=1}^ds_i^{\nu_i+\beta_i}ds_i
\le A_0^nA_w^2\cdot M_n^2 (\frac{1}{2})^{2n}\int_{Q'_j}|g|^2\prod_{i=1}^ds_i^{\nu_i}ds_i.
\]
Otherwise it is \emph{bad}. For a bad cube there exists some \(\beta\) with \(|\beta|>0\) for which the opposite inequality holds, hence
\[
\int_{Q'_j}|\partial^\beta g(s)|^2 \prod_{i=1}^ds_i^{\nu_i+\beta_i}ds_i
\geq A_0^n (A_w\cdot M_n)^2 (\frac{1}{2})^{2n}\int_{Q'_j}|g|^2\prod_{i=1}^ds_i^{\nu_i}ds_i.
\]
Summing the contributions from all bad cubes, we obtain:
	\begin{equation*}
		\begin{aligned}
			\int_{\bigcup Q_j' \text{ is bad}}  |g(x)|^2 \, dm_{\nu}(x) &\leq \sum_{\beta \in \mathbb{N}^d,|\beta|>0 \text{}}A_0^{\,-n} M_n^{-2}2^{2n}A_w^{-2} \int_{\mathbb{R}_+^d}|\partial^\beta g(s)|^2\prod_{i=1}^d s_i^{ \nu_i+\beta_i }ds_i \\
			&\leq \sum_{\beta \in \mathbb{N}^d,|\beta| > 0 } A_0^{-n} \int_{\mathbb{R}^d_{+}} |g(x)|^2 \, dm_\nu(x)\\
			&=   \left(\left(A_0\right)^{d}\left(A_0-1\right)^{-d} -1\right) \int_{\mathbb{R}^d_{+}} |g(x)|^2 \, dm_\nu(x)\\
			&=\frac{1}{2} \int_{\mathbb{R}_+^d} |g(x)|^2 \, dm_\nu(x).
		\end{aligned}
	\end{equation*}
	The final equality is achieved by choosing $A_0=  \frac{1}{1 - (2/3)^{1/d}} $, which ensures that $\left(A_0\right)^{d}\left(A_0-1\right)^{-d} -1=1/2$.

Combining the above estimates with Lemma \ref{L33} and setting $\delta = \left( \frac{1}{10d C_{A,W,\nu}} \right)^{\frac{1}{\min\{\nu_i\} + 1}}$, we obtain a lower bound for the energy concentrated on the good cubes:
          \begin{equation}\label{s5}
\int_{\bigcup_{n \in \text{good}} Q_n} |f|^2 \, d\mu_\nu \ge \frac{9}{10}\|f\|_{L^2_\nu}^2 - \frac{1}{2}\|f\|_{L^2_\nu}^2 = \frac{2}{5}\|f\|_{L^2_\nu}^2.
\end{equation}

\noindent\textbf{Step 3. Quasi‑analyticity on good cubes.}
The rectangles \({Q'_j}\) have pairwise disjoint interiors and cover \(\mathbb{R}_+^d\setminus Z_{\delta^2}\) up to a set of measure zero. For each \(Q'_j\), let
\[
g_{Q'_j}:=g|_{Q'_j}
\]
denote the restriction of \(g\) to \(Q'_j\). First, we recall the elementary Sobolev inequality for cubes with sidelength 1 (see \cite{BJ}):
 \begin{equation}\label{Sobolev}
     \|g\|_{L^\infty(Q_{unit})} \le C(d) \|g\|_{L^2(Q_{unit})} + C(d) \sum_{|\alpha|=d} \|\partial^\alpha g\|_{L^2(Q_{unit})}.
 \end{equation}
Since $Q'$ will no longer be a cube with sidelength 1, the above elementary inequality is no longer tight, to make use of a more tight estimate, we need a rescaling argument. For this purpose, suppose  $Q_j'$ is a rectangular region with side lengths $L_1$, $L_2$, $\dots$, $L_d$. According to the definition $Q_j'=\{x:\sqrt{x_i}\in[j_i+c,j_i+c+1]\}$, the side length in the i-th dimension is calculated as: $L_i=(j_i+c+1)^2-(j_i+c)^2=2(j_i+c)+1\sim 2j+1$. This implies that the volume of $Q_j'$ is $V\sim\prod_{i=1}^{d}(2j_i+1)$. To map the Sobolev inequality from the unit cube $Q_{unit}$ to $Q_j'$, we apply the coordinate transformation $x_i=L_iy_i$, where $y\in Q_{unit}$. Let $\phi (y)=g(Ly)$. Now use the Sobolev inequality \eqref{Sobolev} for $\phi (y)$ we have 
\begin{equation*}
     \|\phi\|_{L^\infty( Q_{unit})} \le C(d) \|\phi\|_{L^2( Q_{unit})} + C(d) \sum_{|\alpha|=d} \|\partial^\alpha \phi\|_{L^2( Q_{unit})}.
 \end{equation*}
The maximum value of the function remains invariant under coordinate scaling. Since the volume element transforms as $dx = V dy$, we have 
$\|\phi\|_{L^2( Q_{unit})} = \frac{1}{\sqrt{V}} \|g\|_{L^2(Q'_j)}$. For a general multi-index $\alpha = (\alpha_1, \dots, \alpha_d)$ with $|\alpha|=d$, the chain rule $\frac{\partial}{\partial y_i} = L_i \frac{\partial}{\partial x_i}$ implies that $\partial_y^\alpha \phi = \left( \prod_{i=1}^d L_i^{\alpha_i} \right) \partial_x^\alpha g$. Let $L^\alpha = \prod_{i=1}^d L_i^{\alpha_i}$. The $L^2$ norm of the derivative transforms as:
\[ \|\partial_y^\alpha \phi\|_{L^2( Q_{unit})} = \left( \int_{ Q_{unit}} |L^\alpha \partial_x^\alpha g|^2 dy \right)^{1/2} = \left( \int_{Q'_j} |L^\alpha \partial_x^\alpha g|^2 \frac{1}{V} dx \right)^{1/2} = \frac{L^\alpha}{\sqrt{V}} \|\partial_x^\alpha g\|_{L^2(Q'_j)}.\]
By substituting these scaling relations into the standard Sobolev inequality on the unit cube $ Q_{unit}$, we obtain
 \[\|g\|_{L^\infty(Q'_j)} \le C(d) \left( \frac{1}{\sqrt{V}} \|g\|_{L^2(Q'_j)} + \sum_{|\alpha|=d} \frac{L^\alpha}{\sqrt{V}} \|\partial_x^\alpha g\|_{L^2(Q'_j)} \right). \]
For every multi-index $\beta$ with $|\beta|=n$, using this Sobolev embedding and the definition of a good cube we obtain
\[
\begin{aligned}
&\|\partial^\beta g_{Q_j'}\|_{L^\infty}\\
&\leq C_d\frac{1}{\sqrt{|Q_j'|}}\left(\left(\int_{Q_j'}| \partial^\beta g|^2 dx\right)^{1/2}+\sum_{|\alpha|=d}L^\alpha\left(\int_{Q_j'}|\partial^{\beta+\alpha}g|^2 dx\right)^{1/2}\right) \\
&\leq C_d\frac{1}{\sqrt{|Q_j'|}}\left(\left(\int_{Q_j'}| \partial^\beta g|^2 dx\right)^{1/2}+\sum_{|\alpha|=d}\left(\int_{Q_j'}|\partial^{\beta+\alpha}g|^2\prod_{i=1}^dx_i^{\alpha_i} dx\right)^{1/2}\right) \\
&\leq \frac{C_d}{L^{\beta}\sqrt{|Q_j'|}\sqrt{\inf_{x\in Q_j'} w(x)}}\left(\left(\int_{Q_j'}| \partial^\beta g|^2 \prod_{x_i}x_i^{\nu_i+\beta_i}dx\right)^{1/2}+\sum_{|\gamma|=n+d}\left(\int_{Q_j'}|\partial^{\gamma}g|^2\prod_{x_i}x_i^{\nu_i+\gamma_i} dx\right)^{1/2}\right) \\
&\leq\frac{C(d)A_wM_{n+d}}{L^{\beta}\sqrt{|Q_j'|}\sqrt{\inf_{x\in Q_j'} w(x)}} \Vert g\Vert_{L^2(Q'_j, m_{\nu})}.
\end{aligned}
\]
Where $C(d)$ depends only on $d$, $w(x)=\prod_{i=1}^{d}x_i^{\nu_i}$.
Then take \(\widetilde{g_{Q'_j}}=\frac{\sqrt{|Q_j'|}\sqrt{\inf_{s\in Q'_j} w(s)} }{C(d)A_w\Vert g\Vert_{L^2(Q'_j, m_{\nu})}}g_{Q'_j}\), then we obtain
\begin{equation}
\|\partial^\beta \widetilde{g_{Q_j'}}\|_{L^\infty}\leq \frac{M_{n+d}}{L^{\beta}}.
\end{equation}
Moreover, we have
\[
\Vert \widetilde{g_{Q'_j}}\Vert_{L^\infty}\geq \frac{\sqrt{\inf_{s\in Q'_j} w(s)} }{C(d)A_w\sqrt{\sup _{s\in Q'_j} w(s)}}\geq C(d,A_w,\delta)>0.
\]

\noindent\textbf{Step 4. Quantitative unique continuation on good cubes.}

To apply the quantitative unique continuation property provided by Lemma \ref{LQA}, we rescale the function $\widetilde{g}_{Q'_j}$. Specifically, let $x = \phi(y)$ denote the affine transformation defined by $x_i = a_i + L_i y_i$, where $y = (y_1, \dots, y_d) \in [0,1]^d$ and $a_i$ is the lower bound of the $i$-th coordinate of $Q'_j$. Under this mapping, we define the corresponding function on the unit cube $[0,1]^d$ as $\widetilde{\phi}_{Q'_j}(y) = \widetilde{g}_{Q'_j}(\phi(y))$.

By the chain rule and the derivative estimates established in the previous step, it follows that:
\begin{equation}
\|\partial_y^\beta \widetilde{\phi}_{Q_j'}\|_{L^\infty([0,1]^d)} = L^\beta \|\partial_x^\beta \widetilde{g}_{Q_j'}\|_{L^\infty(Q'_j)} \leq M_{n+d}.
\end{equation}
Furthermore, the construction of $\widetilde{g}_{Q'_j}$ ensures that the $L^\infty$ norm is bounded away from zero:
\begin{equation}
\|\widetilde{\phi}_{Q_j'}\|_{L^\infty([0,1]^d)} = \|\widetilde{g}_{Q'_j}\|_{L^\infty(Q'_j)} \geq C(d, A_w, \delta) > 0.
\end{equation}
Let $E'$ be a measurable set satisfying the relative measure condition $|E' \cap Q'_j| \geq \gamma |Q'_j|$. Under the affine transformation, the set $\widetilde{E}' = \{ y \in [0,1]^d : a + Ly \in E' \}$ satisfies $|\widetilde{E}'| \geq \gamma$ in the unit cube. Applying Lemma \ref{LQA} to $\widetilde{\phi}_{Q_j'}$, we obtain:
\begin{equation}
\|\widetilde{\phi}_{Q_j'} \mathbf{1}_{\widetilde{E}'}\|_{L^1([0,1]^d)} \geq c(\mathscr{M}, d, \gamma, A_w, \delta) > 0.
\end{equation}
To translate this back to the original coordinates, we note that $dx = |Q'_j| dy$. Thus:
\begin{equation}
\int_{Q'_j} |\widetilde{g}_{Q'_j}(x)| \mathbf{1}_{E'}(x) \, dx = |Q'_j| \int_{[0,1]^d} |\widetilde{\phi}_{Q_j'}(y)| \mathbf{1}_{\widetilde{E}'}(y) \, dy \geq |Q'_j| c(\mathscr{M}, d, \gamma, A_w, \delta).
\end{equation}
Given that the weight $w(x)$ varies minimally on each $Q'_j$ (specifically, $w(x) \approx \inf_{Q'_j} w$), and because for quasi-analytic classes the $L^1$ and $L^2$ norms are locally comparable, we transform the estimate back to the weighted $L^2$ norm of $g$:
\begin{equation}
\|g\|_{L^2(E' \cap Q'_j,  m_{\nu})} \geq c(\mathscr{M}, d, \gamma, A_w, \delta) \|g\|_{L^2_(Q'_j, m_{\nu})}.
\end{equation}
Finally, we perform the summation over all good cubes. By the property of good cubes defined in \eqref{s5}, and noting that the constant $\delta$ depends only on $W, A_w$, and $\nu$, we conclude:
\begin{equation}
\begin{aligned}
\|g\|_{L^2(\mathbb{R}_+^d, m_{\nu})}^2 &\leq C \sum_{Q'_j \,\,is\,\, good} \|g\|_{L^2_(Q'_j, m_{\nu})}^2 \\
&\leq C \sum_{Q'_j \,\,is\,\, good} c^{-2} \|g \mathbf{1}_{E' \cap Q'_j}\|_{L^2(\mathbb{R}_+^d, m_{\nu})}^2 \\
&\leq C(W, d, A_w, \lambda, \nu) \|g \mathbf{1}_{E'}\|_{L^2(\mathbb{R}_+^d, m_{\nu})}^2,
\end{aligned}
\end{equation}
which yields the desired global unique continuation estimate \eqref{UN1}.
\par \textbf{Necessity part:} 
Following the approach in \cite{BJ}, we now demonstrate that the condition 
\begin{equation*}
    \int_{0}^{\infty} \frac{\log W(t)}{1 + t^2} \, dt = \infty
\end{equation*}
is necessary for the unique continuation property \eqref{Ucontinuation} to hold. Let $\mu_n := M_{n-1}/M_n$; then $\{\mu_n\}$ is a non-increasing sequence satisfying $\mu_n \leq 1$. 

Suppose that the integral condition fails, i.e.,
\begin{equation*}
    \int_{0}^{\infty} \frac{\log W(t)}{1 + t^{2}} \, dt < \infty.
\end{equation*}
According to \cite[Proposition 2.2]{BJ}, this implies that $\sum_{n} \mu_{n} < \infty$. We employ a Paley–Wiener-type construction to show that there exist functions $f$, supported in arbitrarily small cubes, such that 
\begin{equation*}
    \int_{\mathbb{R}_+^d} |\mathcal{H}_{\nu} f(\xi)|^{2} W(|\xi|)^{2} \, d\mu_\nu(\xi) < \infty,
\end{equation*}
where $d\mu_\nu(\xi) = \prod_{i=1}^d \xi_i^{2\nu_i+1} d\xi_i$.

Fix $\varepsilon > 0$. Choose an integer $n_{0} \geq \max\{100, \lfloor 2|\nu|+10 \rfloor\}$ such that $\sum_{n \geq n_{0}} \mu_{n} < \varepsilon$. We define the Hankel transform of $f$ as
\begin{equation}\label{EPR1}
    \mathcal{H}_{\nu} f(\xi) := (M_{n_{0}-1})^d \prod_{i=1}^d \left( \left[ \frac{\sin((\varepsilon / n_{0}) \xi_i)}{(\varepsilon / n_{0}) \xi_i} \right]^{2n_{0}} \prod_{k \geq n_{0}} \frac{\sin(\mu_{k} \xi_i)}{\mu_{k} \xi_i} \right).
\end{equation}
By construction, $\mathcal{H}_{\nu} f$ is even with respect to each variable $\xi_i$. 

Extending $\xi$ to the complex domain $z \in \mathbb{C}^d$, the Phragmén–Lindelöf principle implies that
\begin{equation*}
    \left| \frac{\sin (\mu_n z_i)}{\mu_n z_i} \right| \leq e^{\mu_n |\operatorname{Im} z_i|},
\end{equation*}
which leads to the growth estimate
\begin{equation*}
    |\mathcal{H}_{\nu} f(z)| \leq (M_{n_{0}-1})^d e^{3\varepsilon \sum_{i=1}^d |\operatorname{Im} z_i|}.
\end{equation*}
Since $n_0 \geq \lfloor 2|\nu|+10 \rfloor$, it follows that $\mathcal{H}_{\nu} f \in L_\nu^2(\mathbb{R}_+^d)$. Applying Lemma~\ref{TPWS} (the Paley–Wiener theorem for the Hankel transform), we conclude that $\operatorname{supp} f \subset [0, 3\varepsilon]^d$.

To estimate the decay of $\mathcal{H}_{\nu} f$, consider $|\xi_i| \geq 1$. For any $n \geq n_0$, let $\ell = n$ in the following product estimate. Using the property $|\sin t| \leq \min\{1, |t|\}$, we have for each $i \in \{1, \dots, d\}$ and any $\ell \geq n_0$:
\begin{equation*}
    M_{n_0-1} \left| \prod_{k \geq n_0} \frac{\sin(\mu_k \xi_i)}{\mu_k \xi_i} \right| \leq M_{n_0-1} \prod_{n_0 \leq k \leq \ell} \frac{1}{\mu_k |\xi_i|} = \frac{M_\ell}{|\xi_i|^{\ell - n_0 + 1}}.
\end{equation*}
Substituting this into \eqref{EPR1}, for $n \geq n_0$, we obtain
\begin{equation*}
    |\mathcal{H}_{\nu} f(\xi)| \prod_{i=1}^d |\xi_i|^n \leq \left(\frac{n_0}{\varepsilon}\right)^{2n_0 d} \prod_{i=1}^d |\xi_i|^n \prod_{i=1}^d \frac{M_n}{|\xi_i|^{n + n_0 + 1}} = \left(\frac{n_0}{\varepsilon}\right)^{2n_0 d} \prod_{i=1}^d \frac{M_n}{|\xi_i|^{n_0 + 1}}.
\end{equation*}
Similarly, for $n \leq n_0$, we have
\begin{equation*}
    |\mathcal{H}_{\nu} f(\xi)| \prod_{i=1}^d |\xi_i|^n \leq (M_{n_{0}-1})^d \left( \frac{n_{0}}{\varepsilon} \right)^{nd} \prod_{i=1}^d \left| \frac{\sin((\varepsilon/n_{0})\xi_i)}{(\varepsilon/n_{0})\xi_i} \right|^{2n_{0}-n} \leq \left(\frac{n_0}{\varepsilon}\right)^{2n_0 d} \prod_{i=1}^d \frac{M_{n_0}}{|\xi_i|^{n_0 + 1}}.
\end{equation*}
Combining these cases, there exists a constant $C = C(n_0, \varepsilon, d)$ such that
\begin{equation*}
    |\mathcal{H}_{\nu} f(\xi)| \leq C \prod_{i=1}^d \frac{\max\{M_n, M_{n_0}\}}{|\xi_i|^{n+n_0+1}}.
\end{equation*}

Finally, appealing to \cite[Lemma 2.1]{BJ}, for $|\xi| = \sqrt{\sum \xi_i^2}$ with $|\xi_i| \geq 1$, the weight $W$ satisfies
\begin{equation*}
    W(|\xi|) \leq \inf_{n} \frac{|\xi|^{n+1}}{M_n} \leq \frac{C'}{|\mathcal{H}_{\nu} f(\xi)|} \prod_{i=1}^d \frac{1}{|\xi_i|^{n_0}}.
\end{equation*}
This bound ensures that
\begin{equation*}
    \int_{\mathbb{R}_+^d} |W(|\xi|)|^2 |\mathcal{H}_{\nu} f(\xi)|^2 \prod_{i=1}^d \xi_i^{2\nu_i+1} \, d\xi_i < \infty,
\end{equation*}
completing the proof that the finiteness of the log-integral precludes the unique continuation property.
\end{proof}
\section{The One-Dimensional Fractal Uncertainty Principle: Theorem \ref{Th1.3}}\label{Sec4}
This section is devoted to the proof of Theorem~\ref{Th1.3}. As a first step, we prove a Beurling--Malliavin type theorem for the Fourier--Bessel transform. This theorem is then applied to derive a unique continuation result for \(L^2\) functions whose spectral support lies in regular sets. With these ingredients at hand, we conclude the proof of Theorem~\ref{Th1.3} by employing the iterative strategy due to Bourgain and Dyatlov \cite{BD}.
\subsection{The Beurling–Malliavin Type Theorem}\label{sec:BM}
In this subsection, we proceed to establish a Beurling–Malliavin-type theorem for the Fourier–Bessel transform by leveraging the fundamental results developed in \cite{CA}. We subsequently apply this theorem to demonstrate that sets with a fractal structure admit suitable damping functions (see Definition \ref{def:damping}).

Let $\omega: \mathbb{R}^d \to \mathbb{R}_{\le 0}$ be a weight function. Following the framework in \cite{CA}, we define the associated growth functions as follows:
\begin{align}
    G(x) &:= \int_{1/2}^{2} |\omega(sx)| \, ds, \label{def:G} \\
    G^*(r) &:= \sup_{|x|=r} G(x). \label{def:Gstar}
\end{align}

In his work \cite{CA}, Cohen obtained two extraordinarily deep results concerning such weights, the first being the plurisubharmonic Beurling–Malliavin lemma (PSH-BM).

\begin{lemma}[PSH-BM]\label{PSHBM1}
    Suppose that $\omega: \mathbb{R}^d \to \mathbb{R}_{\le 0}$ is a weight function satisfying the following conditions:
    \begin{align}
        \omega(x) &= 0, && \text{for } |x| \le 2, \label{LE31} \\
        |\mathscr{D}^\alpha \omega(x)| &\le C_{\mathrm{reg}} \langle x \rangle^{1-|\alpha|}, && \text{for } 0 \le |\alpha| \le 3, \label{LE32} \\
        \int_0^{\infty} \frac{G^*(r)}{1+r^2}\,dr &\le C_{\mathrm{gr}}, \label{LE33}
    \end{align}
    where $C_{\mathrm{reg}}$ and $C_{\mathrm{gr}}$ are positive constants. Then there exists a continuous plurisubharmonic function $u: \mathbb{C}^d \to \mathbb{R}$ such that:
    \begin{align}
        u(x) &\le \omega(x), && \text{for } x \in \mathbb{R}^d, \notag \\
        u(x) &= 0, && \text{for } |x| \le 2, \ x \in \mathbb{R}^d, \label{LE34} \\
        |u(x_1) - u(x_2)| &\le C_{\mathrm{Lip}} |x_1 - x_2|, && \text{for } x_1, x_2 \in \mathbb{R}^d, \label{LE35} \\
        u(x) \le u(x + iy) &\le u(x) + \rho |y|, && \text{for } x + iy \in \mathbb{C}^d. \label{LE36}
    \end{align}
    Moreover, the constants satisfy $C_{\mathrm{Lip}} \le C_d C_{\mathrm{reg}}$ and $\rho \le C_d \max(C_{\mathrm{reg}}, C_{\mathrm{gr}})$, where $C_d$ depends only on the dimension.
\end{lemma}
The second fundamental result is the \textit{analytic Beurling–Malliavin lemma} (A-BM), which constructs a band-limited multiplier from the plurisubharmonic potential.
\begin{lemma}[A-BM]\label{ABM1}
    Let $u: \mathbb{C}^d \to \mathbb{R}$ be a plurisubharmonic function such that $u|_{\mathbb{R}^d} \le 0$ and properties \eqref{LE34}, \eqref{LE35}, and \eqref{LE36} hold. Then there exists an entire function $f: \mathbb{C}^d \to \mathbb{C}$ such that:
    \begin{align}
        |f(x+iy)| &\le A e^{2\rho|y|}, && \text{for some } A > 0, \label{LE37} \\
        |f(x)| &\ge \frac{1}{2}, && \text{for all } x \in B_{r_{\min}}, \label{LE38} \\
        |f(x)| &\le C e^{u(x)}, && \text{for all } x \in \mathbb{R}^d, \label{LE39}\\
        \int_{\mathbb{R}^d}|f(x)|^2dx&<\infty.
    \end{align}
    where constant $r_{\min}=c_d\min(\rho,\rho^{-1})$ and $C= C_de^{C_{Lip}}\max(\rho^{-C_d},e^{2\rho})$.
\end{lemma}
Combining the previous Paley–Wiener theorem (Lemma \ref{TPWS}) with the above PSH-BM lemma and the A-BM lemma yields the following theorem.
\begin{theorem}\label{L330}
    Suppose that $\omega: \mathbb{R}^d \to \mathbb{R}_{\le 0}$ is an even weight function satisfying \eqref{LE31}, \eqref{LE32}, \eqref{LE33}.

    Then for every $\sigma > 0$ and any $\epsilon>0$, there exists a function $f_0 \in L_\nu^2(\mathbb{R}_+^d)$ such that
    \[
    \operatorname{supp}(\mathcal{H}_{\nu} f_0) \subset [0,\sigma+\epsilon]^d,
    \]
 Moreover, there exists $r_{\min} > 0$ with
    \[
    |f_0(x)| \ge \frac{1}{2} \qquad \text{for all } x \in B_+(0,r_{\min}),
    \]
    and a constant $c = c(d,C_{\mathrm{reg}},C_{\mathrm{gr}})$ such that
    \[
    |f_0(x)| \le C e^{\omega(x)} \qquad \text{for all } x \in \mathbb{R}_+^d.
    \]
\end{theorem}
\begin{proof}
Apply Lemmas \ref{PSHBM1} and \ref{ABM1} to obtain an entire function f of the prescribed exponential type satisfying the required majorization. Averaging f over coordinate reflections, we may assume that it is even in each variable. Since the resulting function is nonzero at the origin, the Cauchy estimate yields a uniform lower bound on a smaller ball. Multiplying by an even entire function of arbitrarily small exponential type and sufficiently rapid polynomial decay, if necessary, ensures membership in  $L_\nu^2(\mathbb{R}_+^d)$. The conclusion now follows from the Paley–Wiener theorem, Lemma \ref{TPWS}.
\end{proof}
\subsection{Unique Continuation for $L^2$ Functions with Spectral Support in Regular Sets}
In this subsection, we will use the previous Beurling–Malliavin-type theorem to establish a unique continuation result for $L^2$ functions with spectral support in regular sets. To this end, following \cite{BJ}, we introduce the following definition of damping function.
\begin{definition}\label{def:damping}
    Let $Y \subset \mathbb{R}_+^d$, let $W$ be a unique continuation weight, and let $c_1, c_2 \in (0,1]$. We say that $Y$ admits a $(c_1,c_2,W)$-damping function if there exists $\psi \in L^2_\nu(\mathbb{R}_+^d)$ satisfying
    \begin{align*}
        &\operatorname{supp} \psi \subset B_{+}(c_1), \\
        &|\mathcal{H}_{\nu} \psi(\xi)| \ge c_2 && \text{for } \xi \in B_{+}( c_2), \\
        &|\mathcal{H}_{\nu} \psi(\xi)| \le \langle \xi  \rangle^{-d} && \text{for all } \xi \in \mathbb{R}_+^d, \\
        &|\mathcal{H}_{\nu} \psi(\xi)| \le 1/W(|\xi|) && \text{for all } \xi \in Y.
    \end{align*}
\end{definition}
Inspired by \cite{BD,CA}, we can construct damping functions for regular sets using the Beurling–Malliavin type theorem for the Fourier–Bessel transform. From now on, we focus on one dimension.
\begin{lemma}\label{lem:damping}
    Let $Y \subset [0,h^{-1}]$ be $\delta$-regular with constant $C_R$ on scales \(1\) to \(h^{-1}\). For any $c_1>0$, $Y$ admits a $(c_1, c_2, W)$ damping function where $c_2=c_2(c_1,\delta, C_R,\nu)$ and 
    \[W(t)=\exp \left( \frac{c_3t}{(\log (2+t))^{\alpha}}\right)\]
    for some $\alpha(\delta, C_R)\in (0,1)$ and $c_3(c_1, \delta, C_R,\nu)>0$. 
\end{lemma}
\begin{proof}
Following a construction similar to the Bourgain and Dyatlov's constructions, we can obtain such a damping function, we refer readers to \cite[Lemma 3.1]{BD} or \cite[Proposition 3.5]{CA1} for details.
\end{proof}
\begin{remark}\label{rem:centered-damping}
For each $\eta\geq0$, a damping function for $Y$ centered at
$\eta$ is obtained by applying the preceding construction to
the translated set $Y-\eta$ and its reflection. Translation and reflection preserve the relevant regularity constant, so the construction is uniform in \(\eta\). The resulting even entire multiplier satisfies the required estimates in terms of \(|\xi-\eta|\), and the Fourier--Bessel Paley--Wiener theorem yields the corresponding function \(\psi_\eta\). Thus, no translation invariance of the Fourier--Bessel transform is used.
\end{remark}

We are now ready to state the main result of this subsection.
\begin{proposition}\label{proposition4.1}
Let \(Y\subset[a, a+h^{-1}]\) be $\delta$-regular with constant $C_R$ on scales \(1\) to \(h^{-1}\). Take
\begin{equation}
\mathcal{I}:=\{[j,j+1]|j\in \mathbb{Z}\},
\end{equation}
and assume that for each $I\in \mathcal{I}$, we are given a subinterval $I'\subset I$ with $|I'|=\lambda>0$ independent of $I$.
Define 
\[S=\cup_{I\in \mathcal{I}}I'.\]
Then, there exists a positive constant $a_0(\delta, C_R, \lambda, \nu )$ such that for all $a\geq a_0$ and for any \(f\in L^2_\nu(\mathbb{R}_+)\) with \(\operatorname{supp} \mathcal{H}_{\nu} f\subset Y\), we have
\[
\|f\|_{L^2_\nu}\le C(\lambda,\delta, C_R, \nu)\,\|f\mathbf{1}_S\|_{L^2_\nu}.
\]
\end{proposition}
To prove this proposition, we first recall a variant of \cite[Proposition 3.3]{BD}.
\begin{proposition}\label{propositionBD}
For every $a\geq 0$, assume that $Y=(Y_0+a)\cup (-Y_0-a)$ with $Y_0\subset [0,h^{-1}]$ is $\delta$-regular with constant $C_R$ on scales \(1\) to \(h^{-1}\). Take
\begin{equation}
\mathcal{I}:=\{[j,j+1]|j\in \mathbb{Z}\},
\end{equation}
and assume that for each $I\in \mathcal{I}$, we are given a subinterval $I'\subset I$ with $|I'|=c_0>0$  independent of $I$. Define 
\[S=\cup_{I\in \mathcal{I}}I'.\]
Then there exists $c_1 > 0$ depending only on $\delta$, $C_R$ and $c_0$ such that for all $f\in L^2(\mathbb{R})$ with $\supp \hat{f}\subset Y$, we have 
\begin{equation}
\lVert f\rVert_{L^2(S)}\geq c_1\lVert f\rVert_{L^2(\mathbb{R})}.
\end{equation}
\end{proposition}
\begin{proof}
If $a=0$, this is precisely a corollary of Proposition~3.3 in \cite{BD}. 
For $a\neq 0$, Lemma~2.1 in \cite{BD} implies that for the chosen $a$ the sets $Y_0+a$ and $-Y_0-a$ remain $\delta$-porous from $1$ and $h^{-1}$. 
By slightly modifying the proof of Proposition~3.3 in \cite{BD} 
(or by applying the more direct argument given in Proposition~3.5 of \cite{CA1}), one obtains the desired result; we omit the details here.
\end{proof}
We use Proposition \ref{propositionBD} and the asymptotic expansion of Bessel functions to establish the following key lemma.
\begin{lemma}\label{L5.2}
Let \(Y \subset [a, a+h^{-1}] \subset \mathbb{R}_+\) be $\delta$-regular with constant $C_R$ on scales \(1\) to \(h^{-1}\). Take
\begin{equation}
\mathcal{I}:=\{[j,j+1]|j\in \mathbb{Z}\},
\end{equation}
and assume that for each $I\in \mathcal{I}$, we are given a subinterval $I'\subset I$ with $|I'|=\lambda>0$ independent of $I$.
Define 
\[S=\cup_{I\in \mathcal{I}}I'.\]
 Assume there exists a unique continuation weight \(W\) and a constant \(A > 0\) such that for some \(\eta \in \mathbb{R}_+\):
\begin{equation}
\lVert \mathcal{H}_{\nu}f W(|\eta - \cdot|) \rVert_{L^2_{\nu}(\mathbb{R}_+)} \le A \lVert f \rVert_{L^2_{\nu}(\mathbb{R}_+)}.
\end{equation}
Then there exists a positive constant \(C_0 = C_0(\delta, C_R, \lambda, W, A, \nu)>0\) such that for all \(a \ge C_0\) and all such functions \(f\) with \(\supp \mathcal{H}_{\nu}f \subset Y\), we have:
\begin{equation}
\|f\|_{L_\nu^2(S)} \ge C(\delta, C_R, \lambda) \|f\|_{L_\nu^2(\mathbb{R}_+)},
\end{equation}
where  $C(\delta, C_R, \lambda)>0$ depends only on $\delta$, $C_R$, $\lambda$.
\end{lemma}
\begin{proof}
The proof we give below is inspired by the approach developed in \cite{BRa}. Let $d\mu_\nu(x) = x^{2\nu+1} dx$. We decompose the total norm $\|f\|_{L_\nu^2(\mathbb{R}_+)}^2$ into a small-$x$ part $I_1$ and a large-$x$ part $I_2$ using $\varepsilon > 0$:
\begin{align*}
\|f\|_{L_\nu^2(\mathbb{R}_+)}^2 &= \int_0^\varepsilon |f(x)|^2 d\mu_\nu(x) + \int_\varepsilon^\infty |f(x)|^2 d\mu_\nu(x) =: I_1 + I_2.
\end{align*}
Recall the inverse Hankel transform: $f(x) = \int_{Y} (xy)^{-\nu} J_\nu(xy) \mathcal{H}_\nu f(y) y^{2\nu+1} dy$.

\textbf{Step 1: Estimation of \(I_1\).}
We use the bound $|z^{-\nu}J_\nu(z)| \le C_{\nu,1} z^{-\nu-1/2}$, which holds for $\nu > -1/2$. Substituting this into $f(x)$:
\begin{align*}
|f(x)| &\le C_{\nu,1} \int_{Y} (xy)^{-\nu-1/2} |\mathcal{H}_{\nu} f(y)| y^{2\nu+1} dy \\
&= C_{\nu,1} x^{-\nu-1/2} \int_{Y} |\mathcal{H}_{\nu} f(y)| y^{\nu+1/2} dy.
\end{align*}
Squaring $|f(x)|$ and integrating over $d\mu_\nu(x) = x^{2\nu+1} dx$:
\begin{align*}
I_1 &\le \int_0^\varepsilon \left( C_{\nu,1} x^{-\nu-1/2} \int_{Y} |\mathcal{H}_{\nu} f(y)| y^{\nu+1/2} dy \right)^2 x^{2\nu+1} dx \\
&= C_{\nu,1}^2 \int_0^\varepsilon x^{-2\nu-1} \left( \int_{Y} |\mathcal{H}_{\nu} f(y)| y^{\nu+1/2} dy \right)^2 x^{2\nu+1} dx \\
&= C_{\nu,1}^2 \int_0^\varepsilon \left( \int_{Y} |\mathcal{H}_{\nu} f(y)| y^{\nu+1/2} dy \right)^2 dx = C_{\nu,1}^2 \varepsilon \left( \int_{Y} |\mathcal{H}_{\nu} f(y)| y^{\nu+1/2} dy \right)^2.
\end{align*}
Applying Cauchy-Schwarz with the weight $W(|\eta-y|)$ to the inner integral:
\begin{align*}
\left( \int_{Y} |\mathcal{H}_{\nu} f(y)| y^{\nu+1/2} dy \right)^2 &= \left( \int_{Y} |\mathcal{H}_{\nu} f(y)| y^{\frac{2\nu+1}{2}} W(|\eta-y|) \cdot \frac{1}{W(|\eta-y|)} dy \right)^2 \\
&\le \left( \int_{Y} |\mathcal{H}_{\nu} f(y)|^2 W(|\eta-y|)^2 y^{2\nu+1} dy \right) \left( \int_{Y} \frac{1}{W(|\eta-y|)^2} dy \right) \\
&\le A^2 \|f\|_{L_\nu^2(\mathbb{R}_+)}^2 \cdot C_w,
\end{align*}
where $C_w = \int_{\mathbb{R}} W(s)^{-2} ds < \infty$. Thus, $I_1 \le C_{\nu,1}^2 C_w A^2 \varepsilon \|f\|_{L_\nu^2(\mathbb{R}_+)}^2$. This estimate is independent of $\eta$, $a$ and $h$.

\textbf{Step 2: Estimation of \(I_2\).}
For $xy \ge \varepsilon a \gg 1$, we use the asymptotic expansion $J_\nu(r) = \sqrt{\frac{2}{\pi r}} \cos(r - \theta_\nu) + R_\nu(r)$, where $|R_\nu(r)| \le C_{\nu,2} r^{-3/2}$. Similar to Step 1, $I_2 \le 2(I_3 + I_4)$, where
\[I_3=\int_\varepsilon^{+\infty} \left| \int_{ Y} (xy)^{-\nu} \sqrt{\frac{2}{\pi xy}} \cos\left(xy - \frac{\nu\pi}{2} - \frac{\pi}{4}\right) \mathcal{H}_{\nu} f(y) \, y^{2\nu+1} \, dy \right|^2 x^{2\nu+1} \, dx,\]
\[I_4=\int_\varepsilon^{+\infty} \left| \int_{ Y} (xy)^{-\nu} R_\nu(xy) \mathcal{H}_{\nu} f(y) \, y^{2\nu+1} \, dy \right|^2 x^{2\nu+1} \, dx.\]

For the remainder $I_4$:
\begin{align*}
I_4 &\le \int_\varepsilon^\infty \left( \int_Y C_{\nu,2} (xy)^{-\nu-3/2} |\mathcal{H}_\nu f(y)| y^{2\nu+1} dy \right)^2 x^{2\nu+1} dx \\
&\le C_{\nu,2}^2\int_\varepsilon^\infty x^{-2\nu-3} \left( \int_Y |\mathcal{H}_\nu f(y)| y^{\nu-1/2} dy \right)^2 x^{2\nu+1} dx \\
&\le \frac{C_{\nu,2}^2}{\varepsilon a^2} \|f\|_{L^2_\nu}^2.
\end{align*}

\textbf{Step 3: Fourier-type Integral \(I_3\).}
Substituting $y = a + s$ into the expression for $I_3$, we obtain
\begin{equation*}
I_3 = \sqrt{\frac{8}{\pi}} \int_\varepsilon^{+\infty} \left| \int_{Y'} \cos(x(a+s) - \theta) \, h(s) \, ds \right|^2  \, dx,
\end{equation*}
where $Y' = Y-a$, $\theta = \frac{\nu\pi}{2} + \frac{\pi}{4}$, and $h(s) = \mathcal{H}_{\nu} f(a+s) \, (a+s)^{\nu+\frac{1}{2}}$. We define the kernel function
\begin{equation*}
g(x) = \int_{Y'} \cos(x(a+s) - \theta) \, h(s) \, ds.
\end{equation*}
Applying the expansion $\cos(xa-\theta+xs) = \cos(xa-\theta)\cos(xs) - \sin(xa-\theta)\sin(xs)$, $g(x)$ can be represented as:
\begin{align*}
g(x) &= \cos(xa-\theta) \int_{Y'} h(s) \cos(xs) \, ds - \sin(xa-\theta) \int_{Y'} h(s) \sin(xs) \, ds \\
&= \frac{e^{i(xa-\theta)}}{2} H(x) + \frac{e^{-i(xa-\theta)}}{2}  H(-x),
\end{align*}
where $H(x) = \int_{Y'} h(s) e^{ixs} \, ds$ is the Fourier transform of $h(s)$. Given that $h(s)$ is supported on $Y'$, the Fourier support of $H(x)$ is $Y'$. By the translation properties of the Fourier transform, the support of $\hat{g}$ is contained in $(Y' + a) \cup (-Y' - a)$. Since $Y' + a = Y$, we conclude that
\begin{equation*}
\operatorname{supp} \hat{g} \subset Y \cup (-Y).
\end{equation*}
Thus by Proposition \ref{propositionBD}, we have
\begin{equation}
\|g\|_{L^2(\mathbb{R})}^2 \le C^*(\delta, C_R, \lambda) \|g\|_{L^2(S)}^2.
\end{equation}
Since $I_3=\sqrt{\frac{8}{\pi}}\int_\varepsilon^{+\infty} |g(x)|^2dx$, we get
\[
I_3 \le \sqrt{\frac{8}{\pi}}\|g\|_{L^2(\mathbb{R})}^2 \le\sqrt{\frac{8}{\pi}} C^{*} \|g\|_{L^2(S)}^2.
\]
\textbf{Step 4: Estimating \(\|g\|_{L^2(S)}\).}
We split $\|g\|_{L^2(S)}^2 = \int_{S \cap [-\varepsilon, \varepsilon]} |g|^2 dx + \int_{S \setminus [-\varepsilon, \varepsilon]} |g|^2 dx =: I_5 + I_6$.
$I_5 \le \varepsilon A^2 C_w \|f\|^2$ using the same technique as in $I_1$. For $I_6$, reversing the expansion and recalling the relation between the cosine term and the Bessel function $J_\nu$ via the remainder $R_\nu(xy)$, we write:
\begin{align*}
I_6 &= \frac{\pi}{2} \int_{E \setminus [-\varepsilon, \varepsilon]} \left| \int_0^{\infty} \sqrt{\frac{2}{\pi xy}} \cos(xy-\theta) (xy)^{-\nu} \mathcal{H}_{\nu} f(y) \, y^{2\nu+1} \, dy \right|^2 x^{2\nu+1} \, dx \\
&\leq \pi \int_{E \setminus [-\varepsilon, \varepsilon]} \left| \int_0^{\infty} J_\nu(xy) (xy)^{-\nu} \mathcal{H}_{\nu} f(y) \, y^{2\nu+1} \, dy \right|^2 x^{2\nu+1} \, dx \\
&\quad + \pi \int_{E \setminus [-\varepsilon, \varepsilon]} \left| \int_0^{\infty} (xy)^{-\nu} R_\nu(xy) \mathcal{H}_{\nu} f(y) \, y^{2\nu+1} \, dy \right|^2 x^{2\nu+1} \, dx \\
&=: \pi (I_7 + I_8).
\end{align*}

For the term $I_7$, the inverse formula of the Hankel transform $\mathcal{H}_\nu$ yields:
\[
I_7 \le \|\mathcal{H}_\nu(\mathcal{H}_\nu f)\|_{L_\nu^2(S)}^2 = \|f\|_{L_\nu^2(S)}^2.
\]

For $I_8$, by analogy with the bound derived for $I_4$, there exists a constant $C_{\nu,2}$ such that:
\[
I_8 \le \frac{C_{\nu,2}^2}{\varepsilon a^2} \|f\|_{L_\nu^2(\mathbb{R}_+)}^2,
\]
thus
\begin{equation}
I_6 \le \pi\|f\|_{L_\nu^2(S)}^2 + \frac{\pi C_{\nu,2}^2}{\varepsilon a^2} \|f\|_{L_\nu^2(\mathbb{R}_+)}^2.
\end{equation}

\textbf{Step 5: Final Result.}
Summing up the above inequalities and setting $\varepsilon = 1/a$, we obtain:
\[ \|f\|_{L_\nu^2}^2 \le \frac{1}{a}\left( A^2 C_w(C_{\nu,1}^2+\frac{4\sqrt{2}}{\sqrt{\pi}}C^{*}) + 8\sqrt{\pi}C_{\nu,2}^2C^*\right) \|f\|_{L_\nu^2}^2 + 4\sqrt{2\pi}C^{*} \|f\|_{L_\nu^2(S)}^2. \]
 By setting $C_0 = 2 \left( A^2 C_w \left( C_{\nu,1}^2 + \frac{4\sqrt{2}}{\sqrt{\pi}} C^* \right) + 8\sqrt{\pi} C_{\nu,2}^2C^* \right)$, it follows that for $a \ge C_0$, the first term on the right-hand side can be absorbed into the left-hand side, yielding:
 \[ 
 \|f\|_{L_\nu^2(S)}^2 \ge \frac{1}{8\sqrt{2\pi}C^*} \|f\|_{L_\nu^2(\mathbb{R}_+)}^2. 
 \]
\end{proof}
Now, with Lemma \ref{L5.2} in hand, we are ready to prove Proposition \ref{proposition4.1}.
\begin{proof}[Proof of Proposition \ref{proposition4.1}]
We follow the strategy of \cite[Theorem 5.2]{BJ} (see also \cite[Theorem 3.4]{CA1}). Consider the lattice \(\Lambda=c_2\mathbb{N}\). The factor is chosen so that the \(c_2\)-neighbourhood of \(\Lambda\) covers \(\mathbb{R}_+\). For each $\eta\in\Lambda$, Remark~\ref{rem:centered-damping} yields a damping function $\psi_\eta$ for $Y$ centered at $\eta$; namely,
\[
\begin{aligned}
&\operatorname{supp}\psi_\eta\subset B_{\lambda/4}(0),\\
&|\mathcal{H}_{\nu}\psi_\eta(\xi)|\ge c_2\quad\text{for }\xi\in B_{c_2}(\eta),\\
&|\mathcal{H}_{\nu}\psi_\eta(\xi)|\le\langle\xi-\eta\rangle^{-1}\quad\text{for all }\xi,\\
&|\mathcal{H}_{\nu}\psi_\eta(\xi)|\le 1/W(|\xi-\eta|)\quad\text{for }\xi\in Y.
\end{aligned}
\]
Define \(f_\eta = f *_\nu \psi_\eta\). By Young's inequality and the Plancherel theorem,
\begin{equation}\label{EEE6}
\sum_{\eta\in\Lambda}\|f_\eta\|_{L^2_\nu}^2
\ge\sum_{\eta\in\Lambda}\int |\mathcal{H}_{\nu} f(\xi)|^2|\mathcal{H}_{\nu}\psi_\eta(\xi)|^2 w(\xi)d\xi
\ge c_2^2\|f\|_{L^2_\nu}^2. 
\end{equation}

We call \(\eta\in\Lambda\) \emph{good} if
\begin{equation}\label{EGOOD}
\bigl\|W(|\eta-\cdot|)^{1/2}\mathcal{H}_{\nu} f_\eta\bigr\|_{L^2_\nu}\le A\|f_\eta\|_{L^2_\nu},
\end{equation}
and \emph{bad} otherwise; here \(A\) is a constant to be chosen later. For bad indices we use the pointwise bound \(|\mathcal{H}_{\nu}\psi_\eta(\xi)|\le1/W(|\xi-\eta|)\) to obtain
\[
\|f_\eta\|_{L^2_\nu}^2\le\frac1A^2\bigl\|W(|\eta-\cdot|)^{-1/2}\mathcal{H}_{\nu} f\bigr\|_{L^2_\nu}^2.
\]
Summing over bad \(\eta\) and using that the weight \(W\) grows superpolynomially, we get
\begin{equation}\label{EEE7}
\sum_{\text{bad }\eta}\|f_\eta\|_{L^2_\nu}^2\le\frac{1}{A^2}\,C(W,c_2)\,\|f\|_{L^2_\nu}^2.  
\end{equation}

Now let \(E\subset S\) be the union of the intervals of side \(\lambda/2\) obtained by shrinking each \(I'\) by a factor \(1/2\) around its centre. For a good index \(\eta\), since \(\operatorname{supp} \mathcal{H}_{\nu} f_\eta\subset Y\) whenever \(\operatorname{supp} \mathcal{H}_{\nu} f\subset Y\). By Lemma \ref{L5.2}, there exists a constant $a_0(\delta, W, A,\lambda,\nu)$ such that for any $a\geq a_0$ and any \(f\in L^2_\nu(\mathbb{R}_+)\) with \(\operatorname{supp} \mathcal{H}_{\nu} f\subset Y\), we have
\[
\|f_\eta\|_{L^2_\nu}\le C(\delta,\lambda)\,\|f_\eta\mathbf{1}_E\|_{L^2_\nu}.
\]

Since \(\operatorname{supp}\psi_\eta\subset B_{\lambda/4}(0)\), we have \(f_\eta\mathbf{1}_E = (f\mathbf{1}_S)*_\nu\psi_\eta\) on \(E\). In fact, we can write \[
f_\eta= (f\mathbf{1}_S)*_\nu\psi_\eta+ (f\mathbf{1}_{R_+\backslash S})*_\nu\psi_\eta,
\]
and
\begin{align*}
&(f\mathbf{1}_{R_+\backslash S})*_\nu\psi_\eta\\
&=c_\nu \int_{R_+}\int_{[0,\pi]^d}f(\sqrt{x^2+y^2+2xy\cos \theta})\mathbf{1}_{R_+\backslash S}(\sqrt{x^2+y^2+2xy\cos \theta})\psi_\eta(y)(\sin \theta)^{2\nu}d\theta d\mu_\nu(y).
\end{align*}
When \(x\in E\) and \(y\in B_{\lambda/4}\), then 
\[
|x-y|\leq \sqrt{x^2+y^2+2xy\cos \theta}\leq x+y
\]
so \(\sqrt{x^2+y^2+2xy\cos \theta}\in S\), implying that 
\[
(f\mathbf{1}_{R_+\backslash S})*_\nu\psi_\eta=0.
\]
Therefore
\[
\Vert f_\eta 1_E\Vert_{L_\nu^2}\leq \Vert (f1_S)*_\nu\psi_\eta\Vert_{L_\nu^2}
=\Vert \mathcal{H}_{\nu}(f1_S)\mathcal{H}_{\nu}\psi_\eta\Vert_{L_\nu^2}
\]

Sum over good indices and notice that \(|\mathcal{H}_{\nu}\psi_\eta(\xi)|\leq \langle \xi-\eta\rangle^{-1}\), then 
\[
\sum_{\text{good }\eta}\|f_\eta\|_{L^2_\nu}^2\leq C(\delta,\lambda)^2 \sum_{\text{good }\eta}\Vert \mathcal{H}_{\nu}(f1_S) \mathcal{H}_{\nu}\psi_\eta\Vert_{L_\nu^2}^2\leq C(\delta,\lambda)^2\Vert \mathcal{H}_{\nu}(f1_S)\Vert_{L_\nu^2}^2=C(\delta,\lambda)^2\Vert  f1_S\Vert_{L_\nu^2}^2.
\]

Using \eqref{EEE6} and \eqref{EEE7} we obtain
\[
\|f\|_{L^2_\nu}^2\le \frac{1}{c_2^2}\sum_{\text{good }\eta}\|f_\eta\|_{L^2_\nu}^2 +\frac{C(W,c_2)}{A^2}\|f\|_{L^2_\nu}^2
\le \frac{C(\delta,\lambda)^2}{c_2^2}\|f\mathbf{1}_S\|_{L^2_\nu}^2 +\frac{C(W,c_2)}{A^2}\|f\|_{L^2_\nu}^2.
\]
Setting \(A= \sqrt{2C(W,c_2)}\), we get
\[
\|f\|_{L^2_\nu}\le \frac{2C(\delta,\lambda)^2}{c_2^2}\|f\mathbf{1}_S\|_{L^2_\nu}.
\]
Since $c_2$ and $W$ depend only on $\delta$, $\lambda$ and $\nu$, we get the desired assertion.
\end{proof}
\subsection{Completion of the Proof of Theorem \ref{Th1.3} via an Iterative Approach}\label{sec:iteration}
First, we fix an even, nonnegative Schwartz function $\phi$ on $\mathbb{R}_+$ such that $\operatorname{supp} \mathcal{H}_{\nu} \phi \subset [0,1]$ and $\int_{\mathbb{R}_+} \phi(x) x^{2\nu+1} dx = 1$ (the existence of such a function is well-known). Let $T > 0$ be an integer to be determined later, and define the scaled function
\[
\psi(x) = 2^{2(\nu+1)T} \phi(2^T x).
\]
For each $j \ge 0$, let
\[
\psi_j(x) = 2^{2(\nu+1)j} \psi(2^j x) = 2^{2(\nu+1)(j+T)} \phi(2^{j+T} x).
\]
Now, we define the weight functions $\Psi_n$ for $n \ge 0$ as
\[
\Psi_n = \psi_n *_\nu \mathbf{1}_{X + [0, 2^{-n-T/2})}.
\]
There exists a constant $C_{\phi,\nu}$ depending only on $\phi$ and $\nu$ such that for all $n \ge 0$:
\begin{align}
    \Psi_n(x) &\ge 1 - C_{\phi,\nu} 2^{-T} && \text{for } x \in X, \label{E52} \\
    \Psi_n(x) &\le C_{\phi,\nu} 2^{-T} && \text{for } x \in \mathbb{R}_+ \text{ with } d(x, X) \ge 5 \cdot 2^{-n-T/2}. \label{E53}
\end{align}
In what follows, we provide a brief verification of \eqref{E53}; the argument for \eqref{E52} follows similarly. Using the definition of the Hankel convolution, we have
\begin{align*}
\Psi_n(x) &= \psi_n *_\nu \mathbf{1}_{X + [0, 2^{-n-T/2})} \\
&= c_\nu \int_{\mathbb{R}_+} \int_{0}^{\pi} \mathbf{1}_{X+[0,2^{-n-T/2})}\left(\sqrt{x^2+t^2-2xt\cos \theta}\right)\psi_n (t) (\sin \theta)^{2\nu} d\theta \, d\mu_\nu(t).
\end{align*}

When $x \in \mathbb{R}_+$ satisfies $d(x,X) \ge 5 \cdot 2^{-n-T/2}$ and $t \in  [0, 2^{-n-T/2})$, it follows that $|x-t| \ge 4 \cdot 2^{-n-T/2}$. Since $\sqrt{x^2+t^2-2xt\cos \theta} \ge |x-t|$,   we obtain
\[
\Psi_n(x) \le \int_{\{t\ge 2^{-n-T/2}\}} \psi_n(t)\,d\mu_\nu(t)
=\int_{\{t\ge 2^{T/2}\}} \phi(t)\,d\mu_\nu(t)\leq C_{\phi,\nu}2^{-T}.
\]
This completes the verification. Next, consider any function $f \in L^2_\nu(\mathbb{R}_+)$ satisfying $\operatorname{supp} \mathcal{H}_\nu f \subset Y$. We define the sequence of functions $\{f_m\}_{m \ge 0}$ as follows:
\[
f_m := \left( \prod_{j=0}^{m-1} \Psi_{jT} \right) f \quad \text{for } m \ge 1, \quad \text{and} \quad f_0 := f.
\]
By Lemma~\ref{L15} and a simple induction argument, it can be shown that
\[
\operatorname{supp} \mathcal{H}_\nu f_m \subset \left( Y + \left[ -2^{mT+1}, 2^{mT+1} \right] \right) \cap \mathbb{R}_+.
\]
Let $m_0$ be the largest integer such that $2^{-m_0 T} \ge h$, which is given by
\[
m_0 = \left\lfloor \frac{\log_2 h^{-1}}{T} \right\rfloor.
\]
\begin{lemma}\label{L544}
 There exists \(\gamma_0 = \gamma_0(\delta, C_R, \nu)\in(0,1)\) such that for every \(T\ge T_0(\delta, C_R, \nu)\) and sufficiently large $a(\delta, C_R, \nu)h^{-1}>0$,
\[
\|f_{m+1}\|_{L^2_\nu} \le (1-\gamma_0)\|f_m\|_{L^2_\nu},\qquad 0\le m\le m_0.
\]
\end{lemma}
Assuming the validity of Lemma~\ref{L544} for the moment, we now complete the proof of Theorem~\ref{Th1.3}. By iteratively applying Lemma~\ref{L544} $m_0$ times, we obtain
\begin{equation}\label{eq:5.3}
\|f_{m_0}\|_{L^2_\nu} \le (1-\gamma_0)^{m_0} \|f\|_{L^2_\nu}. 
\end{equation}
Conversely, the lower bound \eqref{E52} implies that
\begin{equation}\label{eq:5.4}
\|f_{m_0}\|_{L^2_\nu} = \biggl\| \biggl( \prod_{j=0}^{m_0-1} \Psi_{jT} \biggr) f \biggr\|_{L^2_\nu} \ge (1 - C_{\phi,\nu}2^{-T})^{m_0} \|f\mathbf{1}_X\|_{L^2_\nu}. 
\end{equation}
Combining \eqref{eq:5.3} and \eqref{eq:5.4}, we find
\begin{equation*}
\|f\mathbf{1}_X\|_{L^2_\nu} \le \left( \frac{1-\gamma_0}{1-C_{\phi,\nu}2^{-T}} \right)^{m_0} \|f\|_{L^2_\nu}.
\end{equation*}
We now choose $T$ sufficiently large such that $1-C_{\phi,\nu}2^{-T} \ge 1-\gamma_0/2$. It then follows that
\begin{equation*}
\frac{1-\gamma_0}{1-C_{\phi,\nu}2^{-T}} \le \frac{1-\gamma_0}{1-\gamma_0/2} \le 1 - \frac{\gamma_0}{2}.
\end{equation*}
Recalling that $m_0 \ge \frac{1}{T}\log_2 h^{-1} - 1$, we finally arrive at the estimate
\begin{equation*}
\|f\mathbf{1}_X\|_{L^2_\nu} \le C h^{\beta} \|f\|_{L^2_\nu},
\end{equation*}
where $\beta = \frac{1}{T} \log_2 (1-\frac{\gamma_0}{2})^{-1}$ (or specifically, $\beta \approx \frac{\gamma_0}{2T \ln 2}$) and the constant $C$ depends on $\gamma_0, C_{\phi,\nu}$, and $T$. This concludes the proof of Theorem~\ref{Th1.3}.
\par It remains to prove Lemma~\ref{L544}.
\begin{proof}[Proof of Lemma \ref{L544}]
By \cite[Proposition~2.10]{DYAS}, the set \(X\) is
\(\delta'\)-porous on scales from \(Ch\) to \(1\), where
\(\delta'>0\) and \(C>0\) depend only on \(\delta\) and \(C_R\);
see \cite[Definition~2.7]{DYAS} for the precise definition of
porosity. For each relevant \(m\), let \(\mathscr I_{mT}\) denote
the collection of intervals of length \(C2^{-mT}\) partitioning
\(\mathbb R_+\). Since \(C2^{-mT}\in[Ch,1]\), for every
\(Q\in\mathscr I_{mT}\), there exists an interval \(J_Q\subset Q\)
such that
\[
|J_Q|=\delta'|Q|=\delta' C2^{-mT},
\qquad J_Q\cap X=\varnothing.
\]
Letting \(x_Q\) be the midpoint of \(J_Q\), we obtain
\[
\operatorname{dist}(x_Q,X)
   \geq \frac{\delta'}{2}|Q|
   =\frac{\delta'C}{2}\,2^{-mT}.
\]
\[
\operatorname{dist}(x_Q, X) \ge C2^{-mT}\delta'.
\]
Set \(\lambda = C2^{-mT}\delta'\) and define the smaller intervals $I_Q = (x_Q + [-\lambda/2, \lambda/2]) \cap \mathbb{R}_+$. Let \(\mathscr{S} = \bigcup_{Q \in \mathscr{I}_{mT}} I_Q\). It follows that \(\operatorname{dist}(\mathscr{S}, X) \ge C2^{-mT}\delta'/2\). 
 For sufficiently large $T$ (depending on $\delta, C_R, \nu$), we ensure by \eqref{E53} that
\begin{equation}\label{Estimate5.3_val}
\Psi_{mT}(x) \le \frac{1}{2} \quad \text{for } x \in \mathscr{S}.
\end{equation}
Consequently, by \eqref{Estimate5.3_val} and the triangle inequality, we have
\begin{equation}\label{E5.4}
\begin{aligned}
\|f_{m+1}\|_{L^2_\nu}^2 &= \|\Psi_{mT}f_m\|_{L^2_\nu}^2 \le \|\Psi_{mT}\|_{\infty}^2\|f_m\mathbf{1}_{\mathscr{S}^c}\|_{L^2_\nu}^2+ \|\Psi_{mT}\mathbf{1}_{\mathscr{S}}\|_{\infty}^2\|f_m\mathbf{1}_{\mathscr{S}}\|_{L^2_\nu}^2 \\
&\le \left( \|f_m\|_{L^2_\nu}^2 - \|f_m \mathbf{1}_{\mathscr{S}}\|_{L^2_\nu}^2\right) + \frac{1}{2} \|f_m \mathbf{1}_{\mathscr{S}}\|_{L^2_\nu}^2 \\
&= \|f_m\|_{L^2_\nu}^2 - \frac{1}{2} \|f_m \mathbf{1}_{\mathscr{S}}\|_{L^2_\nu}^2.
\end{aligned}
\end{equation}
Rescaling $f_m$, we set \(f_m^{\text{resc}}(x) = f_m(2^{-mT}x)\) and \(\mathscr{S}^{\text{resc}} = 2^{mT}\mathscr{S}\). We have
\[
\operatorname{supp} \mathcal{H}_{\nu} f_m^{\text{resc}} \subset \left(2^{-mT}Y + [-2, 2]\right)\cap \mathbb{R}_+.
\]
Since $2^{-mT}\geq h$, applying Proposition~\ref{proposition4.1} to \(f_m^{\text{resc}}\) for $a\geq C(\delta, C_R, \nu)h^{-1}$, we obtain a constant $c = c( \delta, C_R, \nu) > 0$ such that, after rescaling back,
\begin{equation}\label{E5.3}
\|f_m\mathbf{1}_{\mathscr{S}}\|_{L^2_\nu} \ge c \|f_m\|_{L^2_\nu},
\end{equation}
which implies 
\[ \|f_{m+1}\|_{L^2_\nu}^2\leq \left(1-\frac{c( \delta, C_R ,\nu)^2}{2}\right)\|f_m\|_{L^2_\nu}^2\]
as desired.
\end{proof}
\section*{Acknowledgments}
Z. Duan was supported by the National Natural Science Foundation of China under grants 61671009 and 12171178. ChatGPT Plus 5.6 Sol was also used for minor editing and proofreading throughout the paper.

\noindent \textbf{Authors' addresses:}

\medskip
\noindent Xingyu Zhao, School of Mathematics and Statistics, Huazhong University of Science and Technology, Wuhan 430074, P.R. China. \\
Email address: \texttt{zhaoxingyu@hust.edu.cn}

\medskip
\noindent Longben Wei (Corresponding author), School of Mathematical Sciences, Guizhou Normal University, Guiyang 550025, P.R. China. \\
Email address: \texttt{longbenwei51@gmail.com}

\medskip
\noindent Zhiwen Duan, School of Mathematics and Statistics, Huazhong University of Science and Technology, Wuhan 430074, P.R. China. \\
Email address: \texttt{duanzhw@hust.edu.cn}
\end{document}